\documentclass{amsart} 
\usepackage{enumerate}
\usepackage[applemac]{inputenc}
\usepackage{amsmath}
\usepackage{bbm}
\usepackage{latexsym}
\usepackage{xcolor}
\usepackage{tikz}
\usetikzlibrary{arrows}
\usetikzlibrary{patterns}
\usetikzlibrary{shapes}
\usepackage{mathrsfs}
\usepackage{mathtools}
\usepackage{verbatim}
\usepackage{bm}
\usepackage{hyperref}
\hypersetup{
	colorlinks,
	linkcolor={red!50!black},
	citecolor={blue!50!black},
	urlcolor={blue!80!black}
}
\usepackage[mathcal]{euscript}

\newcommand{\R}{\mathbb{R}}

\newcommand{\N}{\mathbb{N}}

\newcommand{\eexp}{\overrightarrow{\exp}}
\newcommand{\leexp}{\overleftarrow{\exp}}

\newcommand{\mc}{\mathcal}
\newcommand{\IM}{\mathrm{Im}}

\newcommand{\iso}{\mathrm{Iso}}

\newcommand{\coker}{\mathrm{coker}}

\newcommand{\dom}{\mathrm{dom}}
\newcommand{\vect}{\mathrm{Vec}}

\newcommand{\lrb}{\left(}
\newcommand{\rrb}{\right)}

\usepackage[usenames,dvipsnames]{pstricks}
\usepackage{epsfig}
\usepackage{pst-grad} % For gradients
\usepackage{pst-plot} % For axes
\DeclareFontFamily{U}{wncy}{}
    \DeclareFontShape{U}{wncy}{m}{n}{<->wncyr10}{}
    \DeclareSymbolFont{mcy}{U}{wncy}{m}{n}
    \DeclareMathSymbol{\Sha}{\mathord}{mcy}{"58}
\newtheorem{thm}{Theorem}[section]
\newtheorem{lemma}[thm]{Lemma}
\newtheorem{cor}[thm]{Corollary}
\newtheorem{prop}[thm]{Proposition}

\newtheorem{defi}[thm]{Definition}

\theoremstyle{definition}
\newtheorem{remark}[thm]{Remark}

\theoremstyle{remark}

\newcommand{\be}{\begin{equation}}
\newcommand{\ee}{\end{equation}}

\usepackage{mathtools} %scommentare per far numerare solo le equationi che sono chiamate nel testo
\mathtoolsset{showonlyrefs,showmanualtags} %scommentare per far numerare solo le equationi che sono chiamate nel testo

\numberwithin{equation}{section}

\renewcommand{\epsilon}{\varepsilon}

\title[third order open mapping]{Third order open mapping theorems and   applications to the end-point map}
\date{\today}



\thanks{The first author has been supported by University of Padova STARS Project ``Sub-Riemannian Geometry and Geometric Measure Theory Issues: Old and
	New"}

\author{Francesco Boarotto}

\address{Dipartimento di Matematica Tullio Levi-Civita,
	Universit\`a   di Padova, Italy}
\email{\href{mailto:francesco.boarotto@math.unipd.it}{\nolinkurl{francesco.boarotto@math.unipd.it}}}

\author{Roberto Monti}
\address{Dipartimento di Matematica Tullio Levi-Civita,
	Universit\`a   di Padova, Italy}
\email{\href{mailto:monti@math.unipd.it}{\nolinkurl{monti@math.unipd.it}}}

\author{Francesco Palmurella}
\address{
ETH Z\"urich Department Mathematik,
R\"amistrasse 101, CH-8093 Z\"urich, Switzerland}

\email{\href{mailto:
francesco.palmurella@math.ethz.ch} 
{\nolinkurl{francesco.palmurella@math.ethz.ch}}}

\begin{document}

	\begin{abstract} This paper is devoted to a third order study of the end-point map in sub-Riemannian geometry. 
	We first prove third order open mapping results for maps from a Banach space
	into a finite dimensional manifold. In a second step, we compute the third order term in the Taylor expansion of the end-point map and 
	we specialize the abstract theory to the study of  length-minimality of sub-Riemannian strictly singular curves. We conclude with the third order analysis of  a specific strictly singular extremal that is not length-minimizing. 
	\end{abstract}
	 
	\maketitle
	
	\section{Introduction}\label{sec:intro}

	\newcommand{\m}{p}

The most challenging open problems in sub-Riemmanian geometry, such as Sard's problem and the regularity of length-minimizing curves, 
are related to our limited understanding of the end-point map, see \cite{Mont,OpProb}. 
	In this work, we extend the analysis of the end-point map  from the second to the third order. In a preliminary part of independent interest, we study  open mapping theorems of the third order for maps from a Banach space into a manifold.

	Let $M$ be a smooth manifold and $\Delta\subset TM$ be a totally non-holonomic (i.e., completely non-integrable) distribution
with rank $2\leq k<\mathrm{dim}(M)$.  
For every point $q_0\in M$, there exist a neighborhood $U\subset M$ of $q_0$ and linearly independent smooth vector fields $f_1,\ldots,f_k \in \mathrm{Vec}(U)$ such that $\Delta =\mathrm{span}\{ f_1,\ldots,f_k\} $ on $U$. The distribution $\Delta$ is non-holonomic (i.e., it satisfies the H\"ormander condition) if 
\be\label{eq:horm}
	\mathrm{Lie}_{q}\{f_1,\dots, f_k\}=T_{q}M \ \ \text{for every }q\in U,
\ee
where   $\mathrm{Lie}_{q}\{f_1,\dots, f_k\}$ denotes the evaluation at $q$ of the Lie algebra generated by $f_1,\dots, f_k$. Given $q\in U$, we say that $\Delta$ has {\em step} $s\in \N$ at $q$ if, to recover the equality in \eqref{eq:horm}, we need Lie brackets of length   $s$ and $s$ is the least integer with this property. We say that $\Delta$ has step $s$ on $U$ if $\Delta$ has step less than or equal to $s$ at every $q\in U$.

We fix on $\Delta$ the  metric that makes $f_1,\ldots,f_k$ orthonormal. A curve $\gamma\in AC([0,1]; 
U)$  is admissible if
$\dot\gamma\in \Delta_\gamma$ a.e.~on $[0,1]$. In this case, we have
\begin{equation} \label{ECCO}
  \dot\gamma =\sum_{i=1}^k u_i f_i(\gamma),\quad \text{a.e.~on $[0,1]$}
\end{equation}
for some unique vector of functions $u\in L^1([0,1];\R^k)$, called the control of $\gamma $. The length of $\gamma$ is $\mathrm{length}(\gamma) :=\| u\|_{L^1([0,1];\R^k)}$. Since our considerations are local around a reference curve $\gamma$, in the sequel we will assume $U=M$.

 Fix a point $q_0\in M$ and let $ X=L^1([0,1];\R^k)$. The end-point map is the map
 $F= F_{q_0}: X \to M$ defined by $F(u) = \gamma(1)$ where $\gamma$ is the unique solution of \eqref{ECCO} such that $\gamma(0)=q_0$.
   The curve $\gamma$
is said to be singular (or abnormal) if the corresponding control $u$ is a critical point of the differential
$d_u F: X \to T_{F(u)} M$, i.e., if the differential is not surjective.
The \emph{corank} of $u$ is the dimension of $T_{F(u)} M / \mathrm{Im}(d_u F)$.
Singular curves do not depend on the metric fixed on $\Delta$ but nontheless they may be length-minimizers. They do not have a counterpart in Riemannian geometry and do not obey the classical Hamiltonian formalism.

The Sard's problem 
investigates the size (dimension, measure, structure) of the set of points of $M$ that are reachable from $q_0$ by singular curves. Even though  Sard's theorem does not hold in infinite-dimensional spaces \cite{Kup_Sard}, it is expected that for the end-point map this set is not too big, see \cite{VittSard,Bel_Sard,Tre_Sard}.

Another important problem is the regularity of length-minimizing curves.
Montgomery first showed in \cite{Montgomeryabn} the existence of smooth {\em strictly} singular curves that are in fact length-minimizing. For the notion of strict singularity we defer to Definition~\ref{defi:stricts}.
For these curves, however, the first order necessary conditions provided by the Pontryagin Maximum Principle \cite{Pontryagin}
do not typically give any further regularity beyond the starting one (Lipschitz or AC). Some  results on the regularity of singular sub-Riemannian geodesics are in \cite{MR2421548,MR3148127,MR3573971,Monti2018,B}, see also
the surveys
\cite{MR,MR3205109}. The difficulty of the problem, again, lies in the complicated structure of the end-point map at critical points.

Similar problems are addressed e.g.~in \cite{CJT,Chit_08}, where the authors study generic properties of singular trajectories, and in \cite{Bon_Kupka,boabar,boasig}, where some regularity results are established for the more general class of control systems affine in the control. A different approach towards singular 
length-minimizing curves can be found in \cite{Boarottoinv,boaler,Ler}, where the authors follow the topological viewpoint rather than the differential one, and study singular curves via homotopy theory and results \`a la Morse. In the case of Carnot groups, singular  
curves are contained in the zero set of specific polynomials, see \cite{MR3077915,MR3853926}.

The second order analysis of the end-point 
map was developed by Agrachev and Sarychev   in \cite{AS96}.  
This  theory  
provides 
necessary conditions for strictly singular length-minimizers. These conditions are deduced from second order open mapping theorems that exploit the notion of regular zero together with Morse's index theory \cite[Chapter 20]{AgraBook}.
This is the starting point of our work.

In a first step, in Section~\ref{sec:thirorder}, we prove abstract third order open mapping theorems for functions $F: X\to M$, where $X$ is a Banach space and $M$ a smooth manifold. In Definition~\ref{defi:thirddiff},
we introduce an intrinsic notion  of  third differential
 $
 \mc{D}_0^3F: \dom(\mc{D}_0^3F) \to \coker (d_0 F) $,
 where $ \dom(\mc{D}_0^3F)\subset \ker(d_0 F)$ is a precise subspace of the kernel of the differential of $F$ at $0\in X$.
Then, we adapt the notion of regular zero to the third differential.
For a given isotropic vector of the second differential
$w_0 \in \mathrm{Iso}(\mc{D}^2_0F)$, in  Definition~\ref{defi:w0regzero}
we introduce the notion of 
 $w_0$-regular zero.

\begin{thm}\label{thm:open_cor_one}
	Let $X$ be a Banach space and $U\subset X$   an open neighborhood of the origin. Let   $F:U\to M$  be a smooth mapping having  a critical point at $0$ of corank $h\geq 1$.  Then:
	
	\begin{itemize}
	 \item [(i)]  If  $h=1$ and  there exists 
	$v\in\dom(\mc{D}_0^3F)$ such that 
	$
		\mc{D}_0^3F(v)\ne 0
	$, 
then $F$ is open at the origin. 

\item [(ii)] For any $h\geq 1$, if there exist $w_0 \in \mathrm{Iso}(\mc{D}^2_0F)$ and $v_0\in \dom(\mathcal D_0^3 F)$ such that $v_0$ is a $w_0$-regular zero for $\mc{D}^3_0 F$, then $F$ is open at the origin. 
	\end{itemize}	
\end{thm}

The first statement is proved in  
Section~\ref{sec:parti}, while the latter 
 is 
shown in Section~\ref{sec:partii}. Notice that the two statements are different in nature: indeed the first one does not use the notion of regular zero. Also, point (ii)  
can be seen as a more geometric version of the third order open mapping theorem  proved by Sussmann in \cite{SusOpen}. 
Its rephrasement  in   
 algebraic terms can be found in 
Theorem~\ref{thm:regzerogen}.
However, this algebraic version  is less satisfactory than its second order counterpart, where the notion of index of a quadratic form produces  conditions that can be applied effectively to the end-point map. In our case, finding sufficient conditions of the algebraic type ensuring the existence of a regular zero for a vector valued cubic map (polynomials of degree $3$) seems a difficult task.

 In Section~\ref{sec:toan}, we use tools of chronological calculus to compute the third order term in the Taylor expansion of the end-point map, see Proposition~\ref{prop:thirdiffendp}. In fact, our procedure is algorithmic and can be used, in principle, to compute also higher order terms. We shall see that, differently from the second order, the representation of the third differential in terms of Lie brackets is not unique. However, the scalarizations onto the cokernel of the first  
differential are uniquely defined.  

 Theorem \ref{thm:open_cor_one}
  and the formula for the third differential of the end-point map
    yield  the following  necessary condition
satisfied  by any adjoint curve of a singular 
length-minimizing trajectory $\gamma$ of corank $1$. The construction of adjoint curves is recalled in Section~\ref{QUA}.
We denote by $d_ u F$ the differential of the end-point map $F:L^1([0,1];\R^k)\to M$ starting from $\gamma(0)=q_0$, and computed at the point $u\in L^1([0,1];\R^k)$, the control of $\gamma$.

	\begin{thm}\label{thm:pointwisecondBIS}
	Let $(M,\Delta,g)$ be a sub-Riemaniann manifold with $\Delta = \mathrm{span}\{ f_1,\ldots, f_k\}$ for $f_1,\ldots, f_k\in \mathrm{Vec}(M)$. Assume that:
	\begin{itemize}
	 \item [(i)] $\gamma:[0,1]\to M $ is a strictly singular length-minimizing curve  of corank $1$;
	 \item [(ii)] the domain $\dom(\mc{D}_u^3F )$ is of finite codimension in $\ker(d_u F)$.
	\end{itemize}
Then any adjoint curve $\lambda:[0,1]\to T^*M$ satisfies,  
	for every  $t\in [0,1]$ and for every $i,j,\ell =1,\ldots, k$,  
	\be \label{result}
		\big\langle \lambda(t), [ f_{i},[f_{j},f_{\ell} ]](\gamma (t))\big\rangle+\big\langle \lambda(t),[ f_{\ell},[ f_{j},f_{i} ]](\gamma (t))\big \rangle =0.
	\ee
	\end{thm}

This result is proved in Section~\ref{QUA}. 
 Notice the nontrivial  assumption (ii) on the dimension of the domain of the third differential. Condition \eqref{result} is the extension to the third order of the
first and second order 
  necessary conditions for length-minimality.
In fact, if $\gamma$ is a corank-one singular length-minimizing curve with adjoint curve $\lambda$, then 
by  the Pontryagin Maximum Principle  we have 
$\langle \lambda , f_j\rangle =0$ identically along the curve, for every $j=1,\ldots, k$.
If in addition $\gamma$ is strictly singular, then $\langle \lambda , [f_i,f_j]\rangle =0$ identically along $\gamma$, for every $i,j=1,\ldots, k$.
This is known 
as Goh condition, see \cite{Goh}.

In Section \ref{sec:example}
 we show an application of the general theory to a specific example of singular curves. We recall the notion of {\em extremal} curves: a horizontal curve $\gamma$ is extremal if it has an adjoint curve $\lambda$ that satisfies the Pontryagin Maximum Principle. A length-minimizing curve is an extremal, but the viceversa needs not   hold. The notion of (strict) singularity applies to extremal curves as well, see Definition~\ref{defi:stricts}.
 
 \begin{thm}\label{thm:example}
 Consider on $M =\R^3$ the distribution $\Delta =\mathrm{span} \{f_1,f_2\}$, where  
\begin{align}\label{eq:vfields}
			f_1=\frac{\partial}{\partial x_1} \quad \textrm{and}\quad
			f_2=(1-x_1)\frac{\partial}{\partial x_2}+x_1^p\frac{\partial}{\partial x_3},
		\end{align}	
and $p\in \N$. Fix on $\Delta$ the metric $g$ that makes $f_1$ and $f_2$ orthonormal. Then:
	\begin{itemize}
		 \item [(i)] For any $p\geq 2$  the curve $t\mapsto \gamma(t) = (0,t,0)$  is a strictly singular extremal  
		 in $(\R^3,\Delta)$. 
		  \item [(ii)] If $p$ is an even integer then $\gamma$  is locally length-minimizing   in $(\R^3,\Delta,g)$.
		  \item [(iii)] If $p=3$  then $\gamma$ is  not locally length-minimizing in $(\R^3,\Delta,g)$. 
	\end{itemize}
\end{thm}

Using Theorem \ref{thm:open_cor_one}, or alternatively Theorem \ref{thm:pointwisecondBIS},
we show that when $p=3$ the end-point map is open at the control of $\gamma$. 
For $\m=5,7,\ldots$, the curve $\gamma$  is probably not length-minimizing. To prove this we would need open mapping theorems of order higher  than $3$.

	\section{Third order open mapping theorems}\label{sec:thirorder}
	
	\subsection{Intrinsic third differential}

Let  $(X,\|\cdot\|)$ be a Banach space and
let $U\subset X$  be  an open neighborhood of the origin. 
We consider a smooth mapping
$F:U\to M$, where $M$ is a smooth 
manifold of dimension $m\in\mathbb N$. 
Here and hereafter, by ``smooth'' we always mean ``$C^\infty$-smooth''.

  By fixing  a local chart for $M$
centered at $F(0)$, we may consider the representative
of $F$ in this chart as a map from $U$ to $\R^m$, and accordingly
consider its $k$-th directional derivatives  $d^k_0 F: X \to \R^m$ 
\[
d_0^k  F(v ) :
= \left. \frac{d^k}{dt^k} F(t v)\right|_{t=0},
\quad v \in X.
\]
We denote by $(v_1,\ldots, v_k) \mapsto d_0^kF (v_1,\ldots, v_k)$ the associated $k$-multilinear maps.
Then we may expand $F$ as a Taylor series at $0$:
\[
F(v)=\sum_{j=0}^k
\frac{1}{j!}d_0^jF(v)+o(\|v\|^k).
\]
For $k\geq 2$,
the maps $d_0^k F$ do not behave tensorially and depend
on  the specific choice of the local chart of $M$.

In \cite[Chapter 20]{AgraBook},
the authors  study a  
chart-independent (or ``intrinsic'') notion
of Hessian, by quotienting out the action of the differential.  
Recall   that 0 is a critical point of $F$ if the differential $d_0F:X\to T_{F(0)}M$ is not surjective.  
The cokernel of $d_0F$ is the quotient space 
\[
\coker(d_0F) = T_{F(0)}M/ \IM(d_0F),
\] and the corank of this critical point is its dimension:
$\dim\lrb T_{F(0)}M/ \IM(d_0F) \rrb=\dim(M)-\dim\lrb \IM\lrb d_0F\rrb\rrb$.   
The central definition for the theory is  the following.
 
	\begin{defi}  
The \emph{intrinsic Hessian} of $F$ at $u=0$  is the quadratic map
$\mc{D}_0^2F:\ker(d_0F) \to \coker(d_0F)$ defined by
\[
\quad
 \mc{D}_0^2F(v ):=\pi_{\coker(d_0F)}( d_0^2F( v)),
\]
where $d^2_0 F$ is  
computed with respect to any chart centered at $F(0)$
and $\pi_{\coker(d_0F)}$ is the   projection
onto $\coker(d_0F)$.
\end{defi}

This definition is independent of the chosen chart 
and   
for any  linear form 
\[
\lambda\in \IM(d_0F)^\perp=
\{ \lambda \in T_{F(0)}^*M \mid \lambda (d_0 F(x) )= 0 \textrm{ for all $x\in X$}\},
\]
and any vector  $v\in \ker(d_0 F)$
there holds
\be\label{eq:hessintr}
\lambda \mc{D}_0^2{F}(v )=\mc{L}_V\circ \mc{L}_V(a\circ F)\big|_0,
\ee
where $a\in C^\infty(M)$ is any function such that $d_0a=\lambda$,
$V \in\vect(U)$ is any smooth vector field  such that $V(0)=v$, and $\mc{L}_V$ denotes the Lie derivative along $V$.

We denote by $(v,w)\mapsto \mc{D}_0^2F(v,w)$ the bilinear form associated with the quadratic map $\mc{D}_0^2F(v)$.

\begin{defi}
A \emph{regular zero} for the intrinsic Hessian  $\mc{D}_0^2F$ is an element $v\in\ker(d_0F)$ such that:
\begin{itemize}
\item [(i)] $\mc{D}_0^2F(v)=0$;
\item [(ii)] the linear map $w\mapsto \mc{D}_0^2F(v,w)$ is surjective  from $\ker(d_0F)$ onto $\coker(d_0F)$.
\end{itemize}
\end{defi}

With these notions, the following theorem holds, see \cite[Theorem 20.3]{AgraBook}.

\begin{thm}[Agrachev-Sarychev]  
\label{prop:regzero2}
If the intrinsic Hessian   $\mc{D}_0^2F$ has a regular zero then $F$ is open at the origin.
\end{thm}

\noindent Necessary conditions for the existence of a regular zero can be found in \cite{Aru11,AS96}. Sufficient conditions are given by the Morse-index theory, see \cite{AgraBook}. 
The existence of a regular zero is only a necessary condition for the openness of a quadratic form. For example, the map $Q:\R^2\to \R^2$ defined by $Q(x_1,x_2)=(x_1^2-x_2^2,2x_1x_2)$ does not have nontrivial zeros and, in particular, it has no regular zeros, but nevertheless it is   open.

Our objective is to carry this program over to third-order derivatives
and to deduce third-order sufficient conditions for the map
$F$ to be open at the origin.
We first need to define an ``intrinsic'' third differential. 
 Let $P:M\to M$ be any diffeomorphism leaving the point $F(0)$ fixed
and let  $\phi:(-\epsilon,\epsilon)\to U$ be a smooth curve such that
$\phi(0)=0$. 
Let us fix a  local chart for $M$ centered at $F(0)$. Here and hereafter, we assume that $F(0)=0$.
Then, locally in this chart,  we have 
\begin{equation}
\label{eq:thirddiffeps}
\begin{split}
\frac{d^3}{d\epsilon^3}P(F(\phi(\epsilon)))\big|_{\epsilon=0}
&
=d_0P\left(d_0^3F(\dot\phi )+3d_0^2F(\ddot\phi,\dot\phi)+d_0F(\dddot\phi)\right)
\\
&
\quad 
+3d_0^2P\left(  d_0^2F(\dot\phi )+d_0F(\ddot\phi),d_0F(\dot\phi) \right)
\\
&
\quad 
+d^3_0P\left(  d_0F(\dot\phi),d_0F(\dot\phi),d_0F(\dot\phi) \right).
\end{split}
\end{equation}

\noindent
The third derivative in the left hand-side of   \eqref{eq:thirddiffeps}
transforms on $T_0M$ as a tangent vector (i.e., according  to the first differential $d_0P$ only)  as soon as $\dot\phi\in\ker(d_0F)$.  
Moreover, a good  definition of the third differential should only depend   tensorially  on \emph{tangent vectors}. 
This means that  the third derivative
\be
	\frac{d^3}{d\epsilon^3}F(\phi(\epsilon))\big|_{\epsilon=0}=d_0^3F(\dot\phi )+3d_0^2F(\ddot\phi,\dot\phi)+d_0F(\dddot\phi)
\ee
should only depend on $\dot{\phi}$. This happens when $d_0^2 F(\ddot\phi,\dot\phi)=0$ modulo the  image of $d_0 F$.
  These considerations motivate the following definition.

\begin{defi}[Intrinsic third differential]
\label{defi:thirddiff} 
Let $F:U\to M$ be a smooth map.
The \emph{domain of the third   differential} of $F$ 
at $u=0$ is:
	\be\label{eq:setsthird}
		\dom(\mc{D}_0^3F):= 
		\big\{ v\in\ker( d_0F )\mid \pi_{\coker(d_0 F)}\lrb d_0^2F(v,x)  \rrb=0 \,\text{ for all } x\in X  \big \}.
	\ee
	where $d^2_0 F$ is  
computed with respect to any chart centered at $F(0)$.
The \emph{third differential} of    $F$  at   $u=0$ is the cubic map
$\mc{D}_0^3F: \dom(\mc{D}_0^3F) \to \coker(d_0 F)  $  defined by 
\be
		\mc{D}_0^3F(v):=\pi_{\coker(d_0 F)}\big( d_0^3F(v ) \big),
	\ee 
where $d^3_0 F$ is  
computed with respect to any chart centered at $F(0)$
and $\pi_{\coker(d_0F)}$
is the  projection onto $\coker(d_0F)$.
\end{defi}

 \begin{remark}Similarly to the Hessian, these definitions do not depend on the chosen chart.
In particular we stress that, while $d_0^2 F$ depends on the chart, the condition $\pi_{\coker(d_0 F)}\left(d_0^2F(v,x) \right)=0$   for all $ x\in X$ is independent 
of this choice.

To see this, we proceed similarly as in \eqref{eq:thirddiffeps}, and we consider smooth curves $\phi,\psi:(-\varepsilon,\varepsilon)\to U$ such that $\phi(0)=\psi(0)=0$, $\dot{\psi}=v\in \ker(d_0F)$ and $\dot{\phi}=x\in X$. Also, we consider $P:M\to M$ to be any diffeomorphism fixing $F(0)=0$ and we fix a local chart around $0$. Then, by polarization, it is not difficult to see that
\be
	\begin{aligned}
		 \left.\frac{d^2}{d\varepsilon^2}P\left(\frac{F(\phi(\varepsilon)+\psi(\varepsilon))-F(\phi(\varepsilon)-\psi(\varepsilon))}{4}\right)\right|_{\varepsilon=0} 
		= d_0P\left( d_0^2F(v,x)-d_0F\left( \frac{\ddot{\psi}}{2} \right) \right),
	\end{aligned}
\ee
and our assertion follows.
\end{remark} 
 
As for $\mc{D}_0^2F$, see \eqref{eq:hessintr}, for every non-zero linear form $\lambda\in \IM(d_0F)^\perp$ 
and every vector $v\in \dom(\mathcal{D}_0 F)$,
there holds 
\[
\label{eq:thirddd}
\lambda \mc{D}_0^3F(v)=\mc{L}_V\circ \mc{L}_V\circ \mc{L}_V (a\circ F)\big|_0,
\]
where $a\in C^\infty(M)$ is any function such that $d_0a=\lambda$,
$V \in\vect(U)$ is any smooth vector field  such that $V(0)=v$,
and $\mathcal{L}_V$ denotes the Lie derivative along $V$.
Indeed, since by assumption we have $d_0F(v)=0$, the identity   
\begin{align}
		\mc{L}_V\circ \mc{L}_V\circ \mc{L}_V (a\circ F)\big|_0&=d_0^3a\left( d_0F(v)  \right)+ 3d_0^2a\left( d^2_0F(v), d_0F(v)  \right)  +d_0a\left( d_0^3F(v) \right)\\
												 &=\lambda d_0^3F(v).
	\end{align}
holds. In particular,  $\mc{L}_V\circ \mc{L}_V\circ \mc{L}_V (a\circ F)\big|_0$ does not depend upon higher order differentials of $a$ at zero.

\subsection{Open mapping at corank-one critical points}\label{sec:parti}

Assume that  $u=0$ is a critical point of $F$ with corank one,   i.e.,  
$\IM(d_0F)^\perp$ is 1-dimensional and for some non-zero linear form $\lambda$ we have $\IM(d_0F)^\perp = \mathrm{span}\{\lambda\}$. 
 To prove point (i) in Theorem~\ref{thm:open_cor_one}, we adapt an idea used in \cite[Lemma 20.1]{AgraBook},
which  consists in finding a suitable  perturbation 
$\phi_\epsilon: X\to X$ with $\phi_\varepsilon(0)=0$,
so that $F\circ\phi_\epsilon$ is open at 0,
thus implying that $F$ is itself open at 0.

 \begin{proof}[{\bf Proof of Theorem~\ref{thm:open_cor_one} - (i)}]
Since $0$ is a corank-one critical point, there exists an $(m-1)$-dimensional subspace $E\subset X$, with $m = \mathrm{dim}(M)$, such that  
$
X=E\oplus \ker(d_0F).
$  Since  $E$ is  isomorphic to $\IM(d_0F)$ via $d_0F$, we identify it with $\R^{m-1}$.
Namely, we  choose a local chart for $M$ centered at
$F(0)$ and we endow $T_{F(0)}M$ with a scalar product
so that we may identify $T_{F(0)}M$ with $\R^m$ and $\IM(d_0F)$ with $\R^{m-1}$ 
We then fix a basis   $(e_i)_{1\le i\le m-1}$  of $E$.

Let $v\in\dom(\mc{D}_0^3F)$ be such that $\mc{D}_0^3F(v)\neq 0$, and let $z_0,z_1\in E$ to be fixed later. For $\epsilon\geq 0$, we define the map
$\phi_\epsilon:\R^{m-1}\times \R\to X$,
\be\label{eq:phi1}
\phi_\epsilon(x,y):=\frac{\epsilon^3 y^3}{3!}v+\frac{\epsilon^6 y^6}{6!}z_0+\frac{\epsilon^9 y^9}{9!}z_1+\frac{\epsilon^9}{9!}x,
\ee
where $x=(x_1,\dots, x_{m-1})\in\R^{m-1}$ is identified with $(x,0)\in\R^m$.
Notice that $\phi_\epsilon(0)=0$ for every $\epsilon>0$. If the composition  $\Phi_\epsilon := F\circ \phi_\epsilon:\R^{m-1}\times \R\to M$ is open at the origin for some $\epsilon>0$, then $F$ is a fortiori open at the origin.
	
We compute the Taylor expansion at zero of  
$\Phi_\epsilon$ with respect to the parameter $\epsilon$. The only non-trivially zero terms in this expansion are: 
\begin{align} 
\label{eq:deps}
 \Phi_\epsilon^{(3)} (x,y)\big|_{\epsilon=0}&= y^3d_0F(v),
 \\
\label{eq:deps2}
\Phi_\epsilon^{(6)}(x,y)\big|_{\epsilon=0}&=y^6\big(  10d_0^2F(v)+d_0F(z_0)  \big) ,
\\
\label{eq:deps3}
 \Phi_\epsilon^{(9)}(x,y)\big|_{\epsilon=0}&=y^9\lrb  280 d_0^3F(v)+84 d^2_0F(v,z_0)+d_0F(z_1)  \rrb+ d_0F(x). 
\end{align}
The term in the first line is zero  since $v\in \ker(d_0F)$. The term in line \eqref{eq:deps2}  is also zero as soon as we choose $z_0\in E$ such that $d_0F(z_0)=-10d_0^2F(v,v)$. This $z_0$ does exist because   $v\in\dom(\mc{D}_0^3F)$ implies that $d_0^2F(v,v)\in \IM(d_0F)$. Finally,   in line \eqref{eq:deps3} we can choose $z_1\in E$ such that 
\[
	d_0F(z_1)=-84 d^2_0F(v,z_0).
\]
This $z_1$ does exist  because, again, $v\in\dom(\mc{D}_0^3F)$ implies that $d^2_0F(v,z_0)\in \IM(d_0F)$.

Eventually, we see that
$
\Phi_\epsilon^{(9)}(x,y)\big|_{\epsilon=0}= 280y^9d_0^3F(v)+ d_0F(x),
$ 
and this implies that   $\Phi_\varepsilon(x,y)$ admits the expansion 
\[
\label{eq:phieps}
	\Phi_\epsilon(x,y)=\frac{\epsilon^9}{9!} \Phi_\epsilon^{(9)}(x,y)\big|_{\epsilon=0}+R_\epsilon(x,y),
\]
where the remainder term $R_\epsilon(x,y)$ is $O(\epsilon^{10})$ as $\epsilon$ tends to zero. Let us define the function $\Psi_\epsilon:\R^{m-1}\times \R \to X$
\begin{align}
	\Psi_\epsilon 
     (x ,y)=\frac{1}{\varepsilon^9}\Phi_\epsilon(x,y^{1/9}). 
\end{align}
Since $\Psi_\epsilon$ is the composition of  $\Phi_\epsilon$ with a  homeomorphism,   $\Phi_\epsilon$ is open at the origin if so is $\Psi_\epsilon$. 
After   a linear change of coordinates the openness at the origin of $\Psi_\epsilon(x,y)$ reduces  to  the openness of 
$
	\widehat \Psi_\epsilon(x,y)=(x,y)+R_\varepsilon(x,y^{1/9}).
$

Given $r>0$, we denote  
by $B_r \subset \R^m$ the ball of radius $r$   centered at the origin.
We show that  there exists $\delta_0>0$ such that $B_{\delta/2} \subset \widehat \Psi_\epsilon(B_\delta )$
for all   $\delta\in (0,\delta_0)$.
This follows from the following claim: 
\[ \label{CLAMMO}
\lim_{(x,y)\to 0}\frac{R_\epsilon(x,y)}{|x|+|y|^9}=0.
\]
In fact,   \eqref{CLAMMO} implies that there exists $\delta_0>0$ such that for all  $0<\delta<\delta_0$ and for all $(x,y)\in B_{\delta} $    we have 
\[
	|R_\epsilon(x,y^{1/9})|\le\frac{1}{2}(|x|+|y|)\le\frac{\delta}{2}.
\]
Then, given $\xi\in B_{\delta/2} $ and letting $\chi^\xi_\epsilon(x,y):=\xi+(x,y)-\widehat \Psi_\epsilon(x,y)$,   the triangle inequality  implies that $\chi^\xi_\epsilon$ maps $B_\delta$ into itself. It follows by the Brouwer's fixed point theorem 
that $\chi^\xi_\epsilon$ has a fixed point in $B_\delta $
for every $\xi\in B_{\delta/2}$, and the openness of $\widehat \Psi_\epsilon$ follows.

We are   left to show   claim \eqref{CLAMMO}. The role of $y$ and $\epsilon$ in \eqref{eq:phi1} is symmetric, and so the partial derivatives $\frac{\partial^k}{\partial y^k}\Phi_\epsilon(x,y)\big|_{(x,y)=0}$ are computed by the chain rule as in \eqref{eq:deps}, switching  $\epsilon$ and $y$.  As a consequence, we have
\[
\begin{split}
&		
\frac{\partial^k}{\partial y^k}\Phi_\epsilon(x,y)\big|_{(x,y)=0}=0, 
\quad 
k=1,\dots,8,
\\
&
\frac{\partial^9}{\partial y^9}\Phi_\epsilon(x,y)\big|_{(x,y)=0}=280 \epsilon^9 d_0^3F(v,v,v).
\end{split}
\]
Similarly, we see that
\[
	\frac{\partial}{\partial x_i}\Phi_\epsilon(x,y)\big|_{(x,y)=0}=\frac{\epsilon^9}{9!}d_0F(e_i),\ \ i=1,\dots, m-1,
\]
and thus we arrive at the expansion
\be\label{eq:expxy}
	\Phi_\epsilon(x,y)=\frac{\epsilon^9}{9!} d_0F(x)+\frac{\epsilon^9}{9!}280 y^9 d_0^3F(v)+O(|x||y|)+o(|x|+|y|^9),
\ee
where the big-O term $O(|x||y|)$ takes care of all the mixed derivatives in $x$ and $y$, up to the tenth order and it satisfies  $O(|x||y|)=o(|x|+|y|^9)$. 
The  
theorem follows. 
\end{proof}

\subsection{Open mapping at critical points of arbitrary corank} \label{sec:partii}

 We turn   to the case of critical points of corank higher than one, and to the proof of point (ii) in Theorem~\ref{thm:open_cor_one}.
We  begin with adapting to the third order setting the notion of regular zero.

\begin{defi}
Let $F:U\subset X\to M$ be a smooth map.
The \emph{isotropy space} of $\mc{D}_0^2F$  is 
\[
	\iso(\mc{D}_0^2F):=\{u\in \ker(d_0F)\mid \mc{D}_0^2F(u)=0\}.
\]
Given an isotropic vector $w_0\in\iso(\mc{D}_0^2F)$,
we define the \emph{second-order image of $F$ at   $w_0$}   as the   subspace  of $\coker(d_0F)$  
\[
\label{eq:Yw0}
\IM(F,2,w_0)=\IM \big(\mathcal{D}_0^2F(w_0,\cdot)).
\]
Finally, we define   the 
\emph{second-order cokernel of $F$ at $w_0$} as  the quotient  
\[
 \coker(F,2,w_0)=	\coker(d_0 F)/ \IM(F,2,w_0).
\]
\end{defi}

 Note that
we have  $\IM(F,2,w_0)=0$ if and only if $w_0\in\ker(\mc{D}_0^2F)=\{ w_0 \in \ker(d_0 F) \mid  \mc{D}_0^2 F(w_0,\cdot) = 0\}$.

\begin{defi}
\label{defi:w0regzero} Let $w_0\in\iso(\mc{D}_0^2F)$.
A \emph{$w_0$-regular zero} for $\mc{D}_0^3 F$
is an element 
$v\in 
\dom(\mc{D}_0^3F)$ 
such that:
\begin{itemize}
\item [(i)]  $\mc{D}^3_0F(v)=0$; 
\item [(ii)] The linear map 
$
\pi_{\coker(F,2,w_0)}(\mc{D}^3_0F(v,v,\cdot)):\dom(\mc{D}_0^3F)\to \coker(F,2,w_0)
$
is surjective. 
\end{itemize}
Above, $\pi_{\coker(F,2,w_0)}$ is the   projection onto $\coker(F,2,w_0)$,
\end{defi}

  \begin{proof}[{\bf Proof of Theorem~\ref{thm:open_cor_one} - (ii)}]
We fix on $T_{F(0)}M$ a scalar product so that we can regard
all the spaces $\coker(d_0 F)$, $\IM \big(\mathcal{D}_0^2F(w_0,\cdot))$
and $\coker(F,2,w_0)$ as subspaces of $T_{F(0)}M$ with direct sums:
\begin{align*}
T_{F(0)}M
&=\IM(d_0F)\oplus \coker(d_0 F),\\
\coker(d_0 F)&=
\IM( \mc{D}_0  ^2F(w_0,\cdot)) \oplus \coker(F,2,w_0).
\end{align*} 
Let $E_1\subset X$, $E_2\subset \dom(\mc{D}_0^2 F) =\ker(d_0 F)$ and $E_3 \subset \dom(\mc{D}_0^3 F)$ be linear subspaces such that the following mappings are linear isomorphisms: 
\[
\begin{split}
& dF_0 : E_1 \to  \IM(d_0F),
\\
& \mc{D}_0^2(w_0 \cdot): E_2\to \IM( \mc{D}_0  ^2F(w_0,\cdot)  ),
\\
& \mc{D}_0^3 (v_0,v_0,\cdot):E_3\to \coker(F,2,w_0).
\end{split}
\] 
We   identify $E_1 = \R^{m_1}$, $E_2 =\R^{m_2}$, and $E_3 = \R^{m_3}$ with $m_1+m_2+m_3= m$ and with coordinates $r\in \R^{m_1}$, $s\in \R^{m_2}$ and $t\in \R^{m_3}$. 
We also identify $r\in \R^{m_1}$ with $r=(r,0,0)\in\R^m$, and similarly for $s$ and $t$. We denote by 
$\bar e_1,\ldots, \bar e_{m_2}$ a basis for $E_2$, and by 
$e_1,\ldots, e_{m_3}$ a basis for $E_3$.

Let  $\nu ,\zeta, 
\xi ,\mu ,\eta_{i}, \xi_{i},\mu_i, \zeta_{\ell},\eta_{ij}$ and $\zeta_{i\ell}$ be points in  $E_1$ to be fixed later.
For $\epsilon>0$ we define the map $\phi_\epsilon:\R^{m_1}\times \R^{m_2}\times \R^{m_3}\to X$ by:
\begin{align}\label{eq:mapmulti}
	\phi_\epsilon(r,s,t)&= 
	\frac{\epsilon^6}{6!} v_0+\frac{\epsilon^7}{7!} t+\frac{\epsilon^8}{8!} w_0+\frac{\epsilon^{11}}{11!}s+\frac{\epsilon^{12}}{12!} \nu +\frac{\epsilon^{14}}{14!} \xi +\frac{\epsilon^{16}}{16!} \mu +\frac{\varepsilon^{18}}{18!} \zeta+\frac{\epsilon^{19}}{19!}r
	\\
	& + \sum_{\ell=1}^{m_2}\frac{\epsilon^{17}}{17!}s_\ell  \zeta_{\ell}	
	+\sum_{i=1}^{m_3}	t_i\Big( 
	\frac{\epsilon^{13}}{13!} 
\eta_i+\frac{\epsilon^{15}}{15!}
 \xi_{i}+   \frac{\epsilon^{19}}{19!} \mu_i    \Big)   		
	\\
	&
	+\frac{\epsilon^{18}}{18!} \sum_{\ell=1}^{m_2}\sum_{i=1}^{m_3}t_is_\ell \zeta_{i\ell}
	+ \frac{\epsilon^{14}}{14!} \sum_{i,j =1}^{m_3}t_i t_j\eta_{ij}  
	.
\end{align}
Then we consider the composition   $\Phi_\epsilon:=F\circ \phi_\epsilon: \R^{m_1}\times \R^{m_2}\times \R^{m_3}\to M$. To prove that $F$ is open at the origin it is sufficient to show that,   for small $\epsilon>0$, $\Phi_\epsilon$ is open at the origin.  

We compute the derivatives of $\epsilon\mapsto \Phi_\epsilon$ and we evaluate them at $\epsilon =0$. We use the short-hand notation $\Phi=\Phi_\epsilon$ and
$\phi=\phi_\epsilon$. The first non-trivially zero derivative at $\epsilon=0$ is the sixth one:
\[
\Phi^{(6)} = F' [\phi^{(6)}]+O(\epsilon), 
\]
that for $\epsilon=0$ gives  $\Phi^{(6)} (0) =   d_0 F(v_0)=0$ because $v_0\in \dom(\mc{D}_0^2 F)\subset \ker(d_0F)$. For $k=7,\ldots,19$ we have
\[
\Phi^{(k)} = F'[\phi^{(k)}] + \sum_{h=1}^{[k/2]} c_{hk} F'' [ \phi^{(h)}, \phi^{(k-h)}]
+ 
\sum_{\substack{ 1\leq h\leq\ell\leq p
\\
h+\ell+p=k
}} 
c_{ h\ell p} F''' [ \phi^{(h)}, \phi^{(\ell)},\phi^{(p)}]+O(\epsilon),
\]
where $c_{hk}$ and $c_{h\ell p}$ are positive integers. For $k=7,\ldots,11$ we have 	
$
\Phi^{(k)} = F' [\phi^{(k)}]+O(\epsilon).
$
The only non-trivially zero cases are $k=7,8,11$, for which  we have $
\Phi^{(k)}(0)=0$. Indeed, for $k=7$ we have $d_0F (t)=0$   because
$t \in E_ 2\subset \ker(d_0F)$;
for $k=8$ we have $d_0F(w_0)=0$ because  
$w_0 \in \mathrm{Iso}(\mc{D}^2_0F)\subset \ker(d_0F)$; for $k=11$ we have $d_0F(s)=0$ because $s\in E_3 \subset \ker(d_0F)$.

For $k=12,\ldots,17$ we have the following expansions:
\[
\begin{split}
\Phi^{(12)} &= F' [ \phi^{(12)}]+ c_{66} 
F''[ \phi^{(6)},\phi^{(6)}]+ 
O(\epsilon)
\\
\Phi^{(13)} &= F' [ \phi^{(13)}]+ c_{67} 
F''[ \phi^{(6)},\phi^{(7)}]+ 
O(\epsilon)
\\
\Phi^{(14)}& = F' [ \phi^{(14)}]+ c_{68} 
F''[ \phi^{(6)},\phi^{(8)}]+  c_{77} 
F''[ \phi^{(7)},\phi^{(7)}]+
O(\epsilon).
\\
\Phi^{(15)}& = F' [ \phi^{(15)}]+ c_{69} 
F''[ \phi^{(6)},\phi^{(9)}]+  c_{78} 
F''[ \phi^{(7)},\phi^{(8)}]+
O(\epsilon)
\\
\Phi^{(16)}& = F' [ \phi^{(16)}]+ c_{6,10} 
F''[ \phi^{(6)},\phi^{(10)}]+  c_{79} 
F''[ \phi^{(7)},\phi^{(9)}]+ c_{88} 
F''[ \phi^{(8)},\phi^{(8)}]+
O(\epsilon)
\\
\Phi^{(17)} &= F' [ \phi^{(17)}]+ c_{6,11} 
F''[ \phi^{(6)},\phi^{(11)}]+  c_{7,10} 
F''[ \phi^{(7)},\phi^{(10)}]+ c_{89} 
F''[ \phi^{(8)},\phi^{(9)}]+
O(\epsilon).   
\end{split}
\]
The equations $\Phi^{(k)}(0)=0$ lead  to the following list of conditions: 
\begin{align}
\label{12}
&  d_0F(\nu )+ c_{66}   d_0^2F(v_0,v_0) =0,
\\
&
\label{13}  d_0F(\eta_i)+  c_{67}   d_0^2F(v_0,e_i )   =0,\quad i=1,\ldots,m_3,
\\&
\label{14}
d_0F(\xi )+   c_{68}   d_0^2F(v_0,w_0 )=0,
\\&
\label{14bis}
d_0 F(\eta_{ij})  +
c_{77}    d_0^2F(e_i,e_j) =0,\quad i,j=1,\ldots,m_3,
\\
&
\label{15}
d_0F(\xi_{i})+        c_{78}    d_0^2F(e_i,w_0)    =0, \quad i=1,\ldots,m_3,
\\
&
\label{16}
d_0 F(\mu ) +       c_{88}    d_0^2F(w_0,w_0)    =0,
\\
&
\label{17}
d_0 F( \zeta_{\ell} )
+ c_{6,11} d_0F^2(v_0, \bar e_\ell )   =0, \quad \ell=1,\ldots,m_2.
\end{align}
Both \eqref{14} and  \eqref{14bis} origin from $\Phi^{(14)}(0)=0$.

Equation  \eqref{12}
has a solution $\nu \in E_1$ because the vector 
$v_0\in \dom(\mathcal D_0^3 F)$ satisfies $\mc{D}_0 ^2 F(v_0)=0$.   Equation 
\eqref{13}  has a solution $\eta_i\in E_1$ 
because, again,  the points $d_0^2F(v_0,e_i ) $ are in the image of the differential.
For the same reason, there exist solutions $\xi , \eta_{ij}, \xi_{i},\mu \in E_1$ of 
\eqref{14}, 
\eqref{14bis},
\eqref{15}, and  
\eqref{16}.

We study  
\eqref{17}.
Since   
$v_0\in \dom(\mathcal D_0^3 F)$ we   have  
$\pi_{\coker(d_0F)} \lrb d_0^2F(v_0,x)  \rrb=0$ for all $x\in X$.  Then $d_0 F^2(v_0,\bar e_\ell)$ also belongs to the image of the differential  and so there exists a solution $\zeta_{\ell}\in E_1$ to \eqref{17}.

Now we consider the cases $k=18,19$. In these cases, the third differential $F'''$ becomes relevant and we have the following expansions:
\[
\begin{split}
\Phi^{(18)} &= F' [ \phi^{(18)}]+
\sum_{k=6}^ 9 c_{k,18-k} F''[\phi^{(k)},\phi^{(18-k)}]
+c_{666} F'''[\phi^{(6)},\phi^{(6)},\phi^{(6)}] +O(\epsilon), 
\\
\Phi^{(19)} &= F' [ \phi^{(19)}]
+
\sum_{k=6}^ 9 c_{k,19-k} F''[\phi^{(k)},\phi^{(19-k)}] 
+c_{667} F'''[\phi^{(6)},\phi^{(6)},\phi^{(7)}] +O(\epsilon).  
\end{split}
\]
The equation  $\Phi^{(18)}(0) =0$  leads to the following conditions: 
\begin{align}
\label{18bis} 
&    d_0F(\zeta)
+ c_{6,12}   d_0^2F(v_0 ,\nu  )  + c_{666} d_0^3F(v_0,v_0,v_0)
=0,
\\ 
\label{18}
&    d_0F(\zeta_{i\ell})
+ c_{7,11}   d_0^2F(e_i ,\bar e _\ell )  
=0 .
\end{align}
We can fix $\zeta,\zeta_{i\ell}\in E_1$ solving \eqref{18bis}, \eqref{18}.
Here we use the fact that $\mc{D}_0^3F(v_0) = 0$.

Finally, we 
require that $\mu_i\in E_1$ solves the equation
\begin{align} 
\label{19} 
& d_0F(\mu_i )
+  c_{6,13}   d_0^2F(v_0,\eta_i )
+ c_{7,12} d_0^2 F(e_i, \nu )  
=0.
\end{align}
In this way we have 
$
\Phi^{(19)}(0) = d_0F (r) 
+ c_{8,11} d_0^2 F(w_0,s)
+c_{667} d_0^3 F(v_0,v_0, t)  $,
so that  
the map $\Phi_\epsilon$ has the following expansion 
\[
	\Phi_\epsilon(r,s,t)=\epsilon^{19}\left( d_0F(r)+ c_{8,11} d_0^2F(w_0,s)+ c_{667}  d_0^3F(v_0,v_0,t) \right)+O(\epsilon^{20}),
\]
with $ c_{8,11} \neq0$ and $c_{667} \neq0$.
It follows that the map   $\Psi
:\R^{m_1}\times \R^{m_2}\times \R^{m_3}\times \R\to M$
\begin{align}\label{eq:psi}
\Psi(r,s,t;\epsilon) =   \epsilon^{-19}\Phi_\epsilon(r,s,t)
\end{align}
is  of class $C^1$,  with $\Psi(0)=0$ and such that the Jacobian $J_{(r,s,t)}\Psi(0)$ is surjective onto $T_0M$.

By the implicit function theorem, there exists $\epsilon_0>0$ and $C^1$-functions $(r,s,t):(-\epsilon_0,\epsilon_0)\to \R^{m_1}\times \R^{m_2}\times \R^{m_3}$ such that, for every $\epsilon\in(-\epsilon_0,\epsilon_0)$
\begin{itemize}
\item[(i)] $\Psi(r(\epsilon),s(\epsilon),t(\epsilon);\epsilon)=0$, and
\item[(ii)] $J_{(r,s,t)}\Psi(r(\epsilon),s(\epsilon),t(\epsilon);\epsilon)$
is surjective  onto $T_0M$. 
\end{itemize}
This proves that $\Psi_\epsilon$ is open at the origin for small $\epsilon>0$, and eventually that $F$ is open at the origin. 
\end{proof}

Theorem~\ref{thm:open_cor_one} - (ii) reduces the open mapping property for $F$ 
 at $0$ to the existence of $w_0\in\iso(\mc{D}_0^2F)$ such that the   third differential 
\[
\mc{D}_0^3F:\dom(\mc{D}_0^3F)\to   \coker(F,2,w_0) 
\]
admits a $w_0$-regular zero,  
and since the manifold $M$ is finite-dimensional,
it is enough to   consider the case when the source space is finite-dimensional.
 
 Let us recall some facts about cubic maps.
Given a  cubic map $P:\R^N\to \R^n$, for integers $n$ and $N$,  
we denote by  $T:\R^N\times \R^N\times \R^N\to \R^n$  the trilinear map associated with $P$.
Then   there hold the following facts:
\begin{itemize}
	\item[(i)] For   $v\in \R^N$, the differential $d_vP:\R^N\to \R^n$ is the linear mapping given by $d_vP(x)=3T(v,v,x)$, for  $x\in\R^N$.
	
	\item[(ii)] For   $v\in \R^N$, the second differential $d^2_vP:\R^N\times \R^N\to \R^n$ is the vector-valued symmetric bilinear form given by $d^2_vP(x,y)=6T(v,x,y)$, for   $x,y\in\R^N$.
	\item[(iii)] The third differential $d^3P:\R^N\times \R^N\times \R^N\to \R^n$ is the vector-valued symmetric trilinear form given by $d^3P(x,y,z)=6T(x,y,z)$, for   $x,y,z\in\R^N$.
	
	\item[(iv)] The third differential defines the linear map $L:\R^N\to\mathrm{Sym}(\R,N)^n $ into the space of $n$-tuples of $N\times N$ symmetric matrices given by 
	\be\label{eq:L}
		L(x)=d^3P(x,\cdot,\cdot)=6T(x,\cdot,\cdot),\ \ \text{for  }x\in \R^N.
	\ee
	We clearly  have the identity $L(x)=d_x^2P$ as  vector-valued symmetric bilinear maps,  for every $x\in \R^N$.
\end{itemize}

\begin{thm}\label{thm:regzerogen}
Let $P:\R^N\to \R^n$ be a cubic map and assume that:
\begin{itemize}
\item [(i)]  $N\ge n+1$;
\item [(ii)] 
  if $e_1,\ldots, e_N$ denotes the canonical basis of $\R^N$, 
for every non-zero $\lambda\in (\R^n)^*$ the   quadratic forms $Q_i^\lambda:\R^N\to \R$, for  $i=1,\dots, N$, 
\begin{align} Q_i^\lambda(x) 
=\lambda d^3P(e_i,x,x) 
\end{align}
do not have common isotropic vectors $x\neq 0$,
 \end{itemize}
then $P$ has a regular zero.
\end{thm}

\begin{proof} Since $N\ge n+1$ and $P$ is a cubic map,   $P$ has a non-trivial zero $v\in \R^N$ by the B\'ezout theorem (see, e.g.,~\cite[Theorem 1, Chapter IV \S 2]{Shaf74}). We claim that this zero is regular.

Suppose by contradiction that $v$ is not regular, i.e.,  there exists a non-zero $\lambda\in (\R^n)^*$ such that
\[
\lambda d_vP(x)=3\lambda T(v,v,x)=0,\ \ \text{for }x\in \R^N.
\]
Denoting by  $\langle\cdot,\cdot\rangle$  the scalar product on $\R^N$, we recall the identity (compare with \eqref{eq:L})
\be\label{eq:identitythird}
\lambda T(u,v,w)=\langle u, \lambda L(v)w\rangle, \ \ \text{for  }u,v,w\in\R^N.
\ee
Cycling the variables in \eqref{eq:identitythird}, we deduce that $\lambda T(v,x,v)=\langle v, \lambda L(x)v\rangle=0$ for every $x\in \R^N$, i.e.,  $v$ is a common isotropic vector for the quadratic forms $L(x)$ as $x$ varies in $\R^N$, which contradicts   (ii).
\end{proof}

\begin{remark}
In the case of scalar cubic maps, that is $P:\R^N\to \R$, Theorem~\ref{thm:regzerogen} can be made more precise. Indeed, if $N\ge 2$, one can prove that the following are equivalent:
\begin{itemize}
	\item[(i)] $P$ has a regular zero.
	\item[(ii)] $P$ is not a perfect cube.
	\item[(iii)] The linear map $L:\R^N\to \mathrm{Sym}(\R,N)$ is of rank strictly greater than one.
\end{itemize}
\end{remark}

We go back to the case of a smooth map $F: X\to M$.

\begin{cor}\label{cor:condition}
Let $X$ be a Banach space, $U\subset X$ a neighborhood of $0\in X$, $M$ a smooth manifold, and $F:U\to M$ a smooth mapping. Assume that there exists $w_0 \in \mathrm{Iso}(\mc{D}^2_0F)$ such that:

\begin{itemize}
\item [(i)] $\dim(\IM(F,2,w_0))+\dim(\dom(\mc{D}_0^3F))>\dim(M)$.
\item [(ii)] For every non-zero $\lambda\in \IM(F,2,w_0)^\perp$ and $v\in \dom(\mc{D}_0^3F)$ there exists $x\in \dom(\mc{D}_0^3F)$ such that
	$
			\lambda \mc{D}_0^3F(v,v,x)\neq 0$.
\end{itemize}
Then $F$ is open at the origin. 

\end{cor}

\begin{proof} We assume without loss of generality that
$\dim(\IM(F,2,w_0))<\dim(M)$ for every $w_0 \in \mathrm{Iso}(\mc{D}^2_0F)$. If there exists 
$w_0 \in \mathrm{Iso}(\mc{D}^2_0F)$ such that
$\dim(\IM(F,2,w_0))=\dim(M)$ then $F$ is open at the origin by   Theorem \ref{prop:regzero2}.  

	By assumptions (i) and (ii), and recalling that $\lambda\in \coker(F,2,w_0)^*=\IM(F,2,w_0)^\perp$, we deduce by Theorem~\ref{thm:regzerogen} that 
	for every non-zero that the mapping $d_0^3F$ has a $\coker(F,2,w_0)$-regular zero $v\in \dom(\mc{D}_0^3F)$.
	Projecting onto $\coker(d_0F)$, we deduce that $v$ is $w_0$-regular zero for $\mc{D}_0^3F$ in the sense of Definition~\ref{defi:w0regzero}, and the claim follows. 
\end{proof}

\begin{remark}\label{rem:unsat} 
	The conclusions of Corollary~\ref{cor:condition} are  unsatisfactory because they are  not easily exploitable in  the study of the end-point map, in particular at critical points of corank higher than one.

	While in the second order analysis  the Morse theory provides, via the algebraic notion of index, 
effective sufficient conditions ensuring the open mapping property, in the third order case we lack 
 a solid algebraic theory describing the invariants of symmetric tensors of order $3$, where not even the concepts of rank and \emph{symmetric} rank necessarily coincide \cite{Shit18}, and the diagonalization process is not clear. This makes it difficult to find effective conditions ensuring the existence of regular zeros for cubic maps.

\end{remark}

\section{Third order analysis of the end-point map}
\label{sec:toan}

In this section
we expand the end-point map and we  compute the precise structure of its third order term.   The computations use the language of chronological calculus for non-autonomous vector fields, that is briefly recalled in the first subsection.

\subsection{Elements of chronological calculus}
\label{sec:chron_calc}

  Let 
 $M$ be a smooth manifold and let 
$ V= (V_t)_{t\in[0,1]}$ be a time-dependent vector field, that is,
a map $M\times[0,1]\to TM$ so that $V(q,t) =V_t(q)\in T_q M$ for every $t\in [0,1]$.

The flow of $V$ 
is the map 
$P:M\times [0,1]\times [0,1]\to M$,
$P(q_0,t_0,t) = P^t_{t_0}(q_0)$, given by evaluating at time $t$ 
the solution to the Cauchy problem:
\begin{align}
\label{eq:Cauchy-pb}
\left\{
\begin{aligned}
\dot q(\tau) &= V_{\tau}(q(\tau)),\\
q(t_0)&=q_0.
\end{aligned}
\right.
\end{align}  
We assume for our purposes that the solution to  \eqref{eq:Cauchy-pb} is defined for every $t\in [0,1]$.  It is enough to assume that
the vector field $V$ is smooth in the space variable 
and locally integrable in the time variable for problem \eqref{eq:Cauchy-pb} to have a unique solution
(see, e.g., \cite[Chapter 2, Theorem 1.1]{CodLev}).
 
We will adopt the point of view
of operatorial calculus. In particular, we interpret  points $q\in M$ as linear functionals on the algebra $C^\infty(M)$, that is as evaluations $q(a) = a(q)$, 
and we interpret  diffemorphisms $B$ of $M$ as
automorphisms of $C^\infty(M)$ defined by the formula  $Ba(q)=a(B(q))$. Finally,   we identify a vector field $V\in \mathrm{Vec} (M)$ with the derivation of the algebra $C^\infty (M)$ given by $a \mapsto V a$. 

The Cauchy problem \eqref{eq:Cauchy-pb} can be reformulated as the following Cauchy problem of operators on $C^\infty(M)$:  
\begin{equation}\label{NUMO}
 \dot P^t_{t_0}  = P^t_{t_0}\circ V_t,\quad P^{t_0}_{t_0} = \mathrm{Id},
\end{equation}
where $\circ$ is the composition of operators on $C^\infty(M)$ acting from left to right, i.e.: 
\be
(q_0\circ \dot P ^t_{t_0}) a = (q_0\circ P^t_{t_0}\circ V_t) a= V_t a(P^t_{t_0}(q_0)) 
\ee
 for every $a\in C^\infty(M)$ and every $q_0\in M$.  
The  characterization \eqref{NUMO} of  $P^t_{t_0}$ 
 motivates the following notation:
\be \label{FLUX}
	\eexp\int_{t_0}^t V_\tau d\tau\, :=P^t_{t_0} ,
\ee
and we call $P^t_{t_0}$ the \emph{right chronological exponential} of $V$.
  Integrating iteratively the differential equation in \eqref{NUMO}, we may formally expand $P^t_{t_0}$ 
in the following Volterra series:
\be
\label{eq:rightcronexp}
	\begin{aligned}
	P^t_{t_0} &=\mathrm{Id}+\sum_{k=1}^\infty \int_{\Sigma_k(t_0,t)} V_{\tau_k}\circ \dots \circ V_{\tau_1}d\tau_k\dots d\tau_1,\ \ &t\ge t_0,\\
	P^t_{t_0} &=\mathrm{Id}+(-1)^k\sum_{k=1}^\infty \int_{\Xi_k(t,t_0)} V_{\tau_k}\circ \dots \circ V_{\tau_1}d\tau_k\dots d\tau_1,\ \ &t< t_0.
	\end{aligned} 
\ee 
where
\be
	\begin{aligned}
		\Sigma_k(t_0,t):=\{(\tau_1,\dots,\tau_k)\in\R^k\mid t_0\le\tau_k\le\dots\le\tau_1\le t\}&\ \ &\text{if }t\geq t_0,\\
		\Xi_k(t,t_0):=\{(\tau_1,\dots,\tau_k)\in\R^k\mid t\le\tau_1\le\dots\le\tau_k\le t_0\}&\ \ &\text{if }t< t_0.
	\end{aligned}
\ee
We  agree that
$\Sigma(0,t) = \Sigma(t)$, $\Xi_k(t,0)=\Xi_k(t)$ and  $\Sigma_k=\Sigma_k(1)$,
that is the $k$-th dimensional simplex.
The series \eqref{eq:rightcronexp}
are to be interpreted as identities of operators on $C^\infty(M)$.
They are never convergent unless $V_t=0$.   
However, considering only finitely many terms
leads to an asymptotic expansion for the chronological exponential
with a precise estimate for the remainder, see \cite[\S 2.4.4]{AgraBook}. 
	
 For fixed $t_0,t \in [0,1]$, $P^t_{t_0}$ is a diffeomorphism of $M$, and we denote its inverse by $Q:M\times[0,1]\times[0,1]\to M$, $Q(q,t_0,t) = Q^t_{t_0}(q)$.
Differentiating in $t$
the operatorial identity $P^t_{t_0}\circ Q^t_{t_0} =\mathrm{Id}$ we obtain $\dot Q^t_{t_0} = - V_t \circ Q^t_{t_0}$, motivating the notation 
\be
\leexp\int_{t_0}^t (-V_\tau) d\tau:=Q^t_{t_0} ,
\ee
and we call $Q^t_{t_0}$ the \emph{left chronological exponential} of $-V_t$.
Notice that in the differential equations for $P^t_{t_0}$ and $Q^t_{t_0}$ the vector fields is composed from the right with $P^t_{t_0}$ and from the left with $Q^t_{t_0}$. Similarly as for $P$, 
the left-chronological exponential 
has the formal expansion:
\be
\label{eq:leftcronexp}\begin{aligned}
Q^t_{t_0}  &=\mathrm{Id}+(-1)^k\sum_{k=1}^\infty
\int_{\Sigma_k(t_0,t)} V_{\tau_1}\circ \dots \circ V_{\tau_k}d\tau_k\dots d\tau_1,\ \ & t\ge t_0,\\
Q^t_{t_0}  &=\mathrm{Id}+\sum_{k=1}^\infty
\int_{\Xi_k(t,t_0)} V_{\tau_1}\circ \dots \circ V_{\tau_k}d\tau_k\dots d\tau_1,\ \ & t< t_0,
\end{aligned}
\ee
and it follows from the definitions that
for any $t_0,t_1\in [0,1]$ there holds the identity:
\begin{equation}
\label{eq:reverse}
	\eexp\int_{t_0}^{t_1}V_\tau d\tau=\leexp\int_{t_1}^{t_0}(-V_\tau) d\tau.
\end{equation}

A tangent vector $v\in T_qM$ can be seen as a linear functional on the algebra $C^\infty(M)$, defined by the formula $v(f)=d_qf(v)$. Given a diffeomorphism $B$ of $M$ we denote by $B_*$ its differential. The tangent vector $B_* v \in T_{B(q) }M$ defines then an operator on $C^\infty(M)$ according to the formula   $B_* v := v\circ B$. Indeed, if $q(t)$ is a differentiable curve such that $q(0)=q$ and $\dot{q}(0)=v$, then for every $a\in C^\infty(M)$ we have:
\be
	(B_*v)a=\frac{d}{dt}a(B(q(t)))\bigg|_{t=0}=\frac{d}{dt}\bigg|_{t=0}\left(q(t)\circ B\right)a=(v\circ B)a.
\ee

Recall next that a  diffeomorphism of $B:M\to M$
acts on tangent vectors and vector fields
via push-forward, namely if $V\in \vect(M)$ we have 
\be
(B_*V)(B(q)) =	B_*(V(q))
\ee
for every $q\in M$. We may interpret this operation in terms of operators
on $C^{\infty}(M)$. The previous identity reads  as the following composition of operators  $q\circ B\circ B_*V = q\circ V\circ B $, that leads to the operatorial definition:
\be
B_* V:= B^{-1}\circ V\circ B.
\ee
For $V\in \vect(M)$ and $B$ a diffeomorphism of $M$, the operator  $(\mathrm{Ad}B)V$ is defined by the formula 
\begin{equation}\label{fox}
(\mathrm{Ad}B)V:= B\circ V\circ B^{-1}. 
\end{equation}
In fact, $\mathrm{Ad}B^{-1}$ acts on vector fields as the push-forward of $B$, and therefore $(\mathrm{Ad}B)V$ coincides with the pull-back of $V$ by $B$.

These notions apply in particular to the maps $\mathrm{Ad}P^t_{t_0}$, allowing for the following  ``infinitesimal'' characterization:
for every $X\in \vect(M)$ there holds
\be\label{eq:derivativeoffield}
	\frac{d}{dt} (\mathrm{Ad}P^t_{t_0}) X
 =(\mathrm{Ad}P^t_{t_0})[V_t,X]
 =:(\mathrm{Ad}P^t_{t_0})\mathrm{ad}(V_t)X,
\ee
where,
by definition, $\mathrm{ad}(Y)X: = [Y, X]$ denotes 
the (left) Lie-bracket as an operator on $\vect(M)$.
Thus, using the argument in \cite[\S 2.5]{AgraBook}, we then see that $\mathrm{Ad}P^t_{t_0}$ is the unique solution to the Cauchy problem on $\vect(M)$ 
\be
	\dot A^t=A^t\circ \mathrm{ad} V_t, \ \ A^{t_0}=\mathrm{Id},
\ee
and this motivates the following notation:
\[
\eexp\int_{t_0}^t\mathrm{ad}V_\tau d\tau:=\mathrm{Ad}\left(\eexp\int_{t_0}^t V_\tau d\tau\right).
\]

\subsection{Expansion of the end-point map}
Let $M$ be a smooth manifold and
let $f_1,\dots, f_k\in \vect(M)$ be smooth vector fields on $M$. Given $u\in
L^1([0,1];\R^k)$ 
we will   use the short-hand notation $f_{u(t)}:=\sum_{i=1}^ku_i(t) f_i$.
Note that $f_u$ is   a time-dependent vector field
as  in the previous subsection.

\begin{defi}
The \emph{end-point map} 
relative to the vector fields
$f_1,\ldots,f_k$ is the map $F:M\times L^1([0,1];\R^k)\to M$ given by
\be
\label{eq:endp}
	F(q_0,u):=F_{q_0}(u):=q_0\circ\eexp \int_0^1 f_{u(t)}\, dt.
\ee
\end{defi}
Recall that we are assuming that the Cauchy problem for $f_{u(t)}$ has a solution defined on the whole interval $[0,1]$.
We perform a perturbation analysis of the end-point
map with respect to the control variable. 
To this aim,
recall that 
by the variation formula in \cite[\S 2.7, (2.28)]{AgraBook} we have:
\begin{align} 
F_{q_0}(u+v)
=q_0\circ\eexp\int_0^1 \big( f_{u(t)}+f_{v(t)} \big) 
dt
= F_{q_0}(u)
\circ \eexp\int_0^1 \mathrm{Ad}\left(\eexp\int_1^t f_{u(\tau)}d\tau \right)f_{v(t)}dt. 
\end{align}
This motivates the following definition.

\begin{defi}
The \emph{perturbation map} relative to the vector fields
$f_1,\ldots,f_k$ is the map
$G:L^1([0,1];\R^k)\times L^1([0,1];\R^k)\times M\to M$
given by	
\begin{equation} \label{ABC}
G(u,v,q_1):= G^u_{q_1}(v)=
q_1
\circ \eexp\int_0^1
\mathrm{Ad}\left(\eexp\int_1^t f_{u(\tau)}d\tau \right)f_{v(t)}dt.
\end{equation}
\end{defi}
The term ``perturbation'' is of course motivated by the fact that,
by the variation formula, there holds:
\begin{align*}
G(u,v,F_{q_0}(u)) = F_{q_0}(u+v),
\end{align*}
so when $q_1=F_{q_0}(u)$ and $v$ is  small,
$G^u_{F_{q_0}(u)}(v)$ is  a small perturbation of $F_{q_0}(u)$.
For $t\in [0,1]$, we define 
the time-dependent vector field
\be\label{eq:gut}
g^{u,t}_{v(t)}:=\mathrm{Ad}\left(\eexp\int_1^t f_{u(\tau)}d\tau \right)f_{v(t)}= 
\eexp\left(\int_1^t\mathrm{ad}f_{u(\tau)}d\tau\right)f_{v(t)}.
\ee
As an operator on $C^\infty(M)$, $G^u_{q_1}(v)$ admits the formal expansion:
\begin{equation}
\label{lio}
G^u_{q_1}(v):=q_1\circ \eexp\int_0^1 g^{u,t}_{v(t)}dt
=q_1\circ \Big(  \mathrm{Id}+\sum_{k=1}^\infty \int_{\Sigma_k}g_{v(\tau_k)}^{u,\tau_k}\circ\dots \circ g_{v(\tau_1)}^{u,\tau_1}d\tau_k\dots d\tau_1\Big).
\end{equation}
 Replacing $v$ by $\varepsilon v$ in \eqref{lio} and dropping the dependence on $q_1$, we introduce the family of diffeomorphisms depending on the parameter $\varepsilon > 0$: 
\be
\label{eq:repad}
	G^u(v\varepsilon)=\mathrm{Id}+\sum_{k=1}^\infty\varepsilon^k\int_{\Sigma_k}g_{v(\tau_k)}^{u,\tau_k}\circ\dots \circ 			g_{v(\tau_1)}^{u,\tau_1}d\tau_k\dots d\tau_1.
\ee

Now we compute a different expansion for $G^u(v\epsilon)$, where the role of the Lie-brackets of $f_1,\ldots, f_k$ is more transparent. We 
can compute the derivative  in $\varepsilon$ of
$G(v\varepsilon)$ 
using \cite[\S 2.8, (2.31)]{AgraBook}:
\be
\label{eq:derivative}
	\frac{\partial}{\partial\varepsilon}G^u(v\varepsilon)=W(v;\varepsilon)\circ G^u(v\varepsilon),
\ee
where the vector field $W(v;\varepsilon)$ is given by the formula 
\be
\label{DELTA}
	W(v;\varepsilon)=\int_0^1\mathrm{Ad}\left( \eexp\int_0^t  \varepsilon g_{v(\tau)}^{u,\tau}d\tau \right)g_{v(t)}^{u,t}dt=\int_0^1\left( \eexp\int_0^t \mathrm{ad} \varepsilon g_{v(\tau)}^{u,\tau}d\tau \right)g_{v(t)}^{u,t}dt.
\ee 
For the definition of the integral $\int_0^1 V_\tau d\tau$ of a non-autonomous vector field $t\mapsto V_t$, we refer to \cite[\S 2.3]{AgraBook}.

Comparing  \eqref{eq:derivative} with    \eqref{eq:leftcronexp} we deduce that:
\be
\label{eq:geps}
G^u(v\varepsilon)=\leexp\int_0^\varepsilon W(v;\eta)d\eta=\mathrm{Id}+\sum_{n=1}^\infty \int_{\Sigma_n(\varepsilon)}W(v;\eta_1)\circ\dots\circ W(v;\eta_n)d\eta_n\dots d\eta_1.
\ee
Thus the formal series in \eqref{eq:repad} and   \eqref{eq:geps} coincide for every $\varepsilon>0$.
From formula \eqref{eq:geps} 
we deduce the following expansion for $G^u(v)$ as an operator on $C^\infty(M)$.
	 
	 \begin{lemma}\label{prop:expbrack}
For every $v\in L^1([0,1];\R^k)$   we have:
\be
\label{eq:expGu}
G^u(v) 
= \mathrm{Id} 
+ d_0G^u(v) 
+\frac {1}{2} d^2_0G^u(v) 
+\frac {1}{6} d^3_0G^u(v)  
+ O(\|v\|_{L^1([0,1];\R^k)}^4),
\ee
where 
\begin{align}
\label{wix} 
d_0G^u(v)& = \int_0^1 g_{v(t)}^{u,t} dt,
\\
\label{wox} 
d^2_0G^u(v)& 
=   \int_{\Sigma_2}[ g^{u,\tau_2}_{v(\tau_2)},g^{u,\tau_1}_{v(\tau_1)} ] d\tau_2 d\tau_1+ 
\left(\int_0^1 g_{v(t)}^{u,t} dt\right)\circ \left(\int_0^1 g_{v(t)}^{u,t} dt\right),
\\
\label{wax} 
d^3_0G^u(v)
&
= 2\int_{\Sigma_3}[ g^{u,\tau_3}_{v(\tau_3)},[ g^{u,\tau_2}_{v(\tau_2)},g^{u,\tau_1}_{v(\tau_1)} ]] d\tau_3 d\tau_2 d\tau_1
\\
&\quad
+2\left(\int_{\Sigma_2}[ g^{u,\tau_2}_{v(\tau_2)},g^{u,\tau_1}_{v(\tau_1)} ] d\tau_2 d\tau_1\right)\circ\left(\int_0^1 g_{v(t)}^{u,t} dt\right) 
\\
&
\quad
+ \left(\int_0^1 g_{v(t)}^{u,t} dt\right)\circ \left(\int_{\Sigma_2}[ g^{u,\tau_2}_{v(\tau_2)},g^{u,\tau_1}_{v(\tau_1)} ] d\tau_2 d\tau_1\right)
\\ &
\quad 
+ \left(\int_0^1 g_{v(t)}^{u,t} dt\right)\circ \left(\int_0^1 g_{v(t)}^{u,t} dt\right)\circ \left(\int_0^1 g_{v(t)}^{u,t} dt\right).
\end{align}
\end{lemma} 
	
	\begin{proof}
		We begin with the expansion of $W(v;\eta)$ in \eqref{DELTA} as a power series in $\eta$. Thanks to \cite[\S 2.5, (2.23)]{AgraBook}, we obtain
		$W(v;\eta)=\sum_{k=1}^\infty \eta^{k-1}W_k(v)$, where
		\be\label{eq:Wterms}
			W_1(v)=\int_0^1 g_{v(t)}^{u,t}dt\ \ \text{and}\ \ W_k(v)=\int_{\Sigma_k}\mathrm{ad}g_{v(\tau_k)}^{u,\tau_k}\circ\dots\circ \mathrm{ad}g_{v(\tau_2)}^{u,\tau_2}(g_{v(\tau_1)}^{u,\tau_1})d\tau_k\dots d\tau_1,\quad k\ge 2.
		\ee
		We then compute the first three terms of the sum in \eqref{eq:geps}:
		\begin{align}
			&\int_{\Sigma_1(\varepsilon)}W(v;\eta_1)d\eta_1=\sum_{h =1}^\infty \frac{\varepsilon^{h }}{h }W_{h }(v),\\
			&\int_{\Sigma_2(\varepsilon)}W(v;\eta_1)\circ W(v;\eta_2)d\eta_2d\eta_1=\sum_{h ,k=1}^\infty\frac{\varepsilon^{h+k}}{(h+k)k}W_{h }(v)\circ W_{k}(v),\\
			&\int_{\Sigma_3(\varepsilon)}W(v;\eta_1)\circ W(v;\eta_2)\circ W(v;\eta_3) d\eta_3d\eta_2d\eta_1
=\sum_{h,k,\ell=1}^\infty\frac{\varepsilon^{h+k+\ell}W_{h}(v)\circ W_{k}(v)\circ W_{\ell}(v)}{(h+k+\ell)(k+\ell)\ell}. 
\end{align}
Then using these formulas in 
\eqref{eq:geps}, 		
we get
\begin{align}
G^u(v\varepsilon)
=\mathrm{Id}&+\varepsilon W_1(v)+\frac{\varepsilon^2}{2}
\big(W_2(v)+W_1(v)\circ W_1(v)\big)
+\frac{\varepsilon^3}{3}
\big(W_3(v)+W_2(v)\circ W_1(v)\big)
\\
& +\frac{\varepsilon^3}{6}
\big( (W_1(v)\circ W_2(v)+W_1(v)\circ W_1(v)\circ W_1(v)\big)+O(\varepsilon^4 
\|v\|_{L^1([0,1];\R^k)}^4
),	
\end{align}
where the estimate on the remainder follows from  Remark~\ref{LEO} below.
From this formula, we can compute the directional derivatives  
\be 
d_0G^u(v)
= \left.\frac{d}{d\epsilon} G^u(v\epsilon)  \right|_{\epsilon=0} ,
\quad  
d^2_0G^u(v) =
\left.\frac{d^2}{d\epsilon^2} G^u(v\epsilon)  \right|_{\epsilon=0} ,
\quad 
d^3_0G^u(v)=
\left.\frac{d^3}{d\epsilon^3} G^u(v\epsilon)  \right|_{\epsilon=0} ,
\ee
obtaining formulas \eqref{wix}, \eqref{wox}, and \eqref{wax}.
\end{proof}
	
\begin{remark} \label{LEO}
Even if Lemma~\ref{prop:expbrack} is enough for our purposes, the computation's method in the proof is algorithmic and permits to determine  the terms of any order in the expansion of $G^u(v\varepsilon)$. Indeed, for $k\ge 1$ we have the formal identity  
\begin{align}
\int_{\Sigma_k(\varepsilon)}W(v;\eta_1)\circ&\dots\circ W(v;\eta_k)d\eta_k\dots d\eta_1= \sum_{h_1,\dots,h_k=1}^\infty \frac{\varepsilon^{h_1+\dots+h_k}W_{h_1}(v)\circ\dots\circ W_{h_k}(v)}{(h_1+\dots+h_k)\dots(h_{k-1}+h_k)h_k}.
\end{align}
\end{remark}

As consequence of Lemma \ref{prop:expbrack}, we obtain
an explicit formula for the intrinsic third differential of $G_{q_1}^u$
(recall Definition \ref{defi:thirddiff}).

\begin{prop}\label{prop:thirdiffendp}
For any   
$v
\in\dom(\mc{D}_0^3G_{q_1}^u)$ and   $\lambda  
\in \IM(d_0G_{q_1}^u)^\perp$ 
we have:  
\be\label{eq:thirdexplicit}
\lambda \mc{D} _0^3 G^u_{q_1}(v)= 
2\int_{\Sigma_3}\left 
\langle \lambda,[ g^{u,\tau_3}_{v(\tau_3)},[ g^{u,\tau_2}_{v(\tau_2)},g^{u,\tau_1}_{v(\tau_1)} ]](q_1) \right\rangle d\tau_3 d\tau_2 d\tau_1.
\ee
\end{prop}

	\begin{proof}
		Let $v\in\dom(\mc{D}_0^3G_{q_1}^u)$ and   $a\in C^\infty(M)$ be such that $a(q_1)=0$ and $d_{q_1}a=\lambda$. 
		Since $\dom(\mc{D}_0^3G_{q_1}^u)\subset \mathrm{ker}(d_0 G_{q_1}^u)$, we deduce that 
\be
\label{NUCL}
	d_0G_{q_1}^u(v)= q_1\circ \int_0^1 g_{v(\tau_1)}^{u,\tau_1} d\tau_1=0.
\ee
By the definition of the third differential and by a computation similar to
\eqref{eq:thirddiffeps}
we have 
\be
\left.\frac{d^3}{d\varepsilon^3} a(G_{q_1}^u(v\varepsilon))\right|_{\varepsilon=0}=\lambda\mc{D}_0^3G_{q_1}^u(v).
\ee
We used 
\eqref{NUCL} to prove that 
the terms involving second and third order derivatives of $a$ are zero.
Moreover, as 
 $v\in\dom(\mc{D}_0^3G_{q_1}^u)$
we also have   
\be
	\left.\frac{d^2}{d\varepsilon^2} a(G_{q_1}^u(v\varepsilon))\right|_{\varepsilon=0}=\lambda\mc{D}_0^2G_{q_1}^u(v)=0.
\ee
Returning to the chronological notation, we have to expand to the third order the expression
\be
(G^u_{q_1}(v ))a=\left(q_1\circ\eexp\int_0^1 g^{u,t}_{v(t)}dt\right) a.
\ee
Comparing \eqref{eq:expGu} with the expansion provided in \eqref{lio}, we have to calculate:
\begin{equation}\label{eq:chronexp}
\begin{aligned}
2&\left(q_1\circ \int_{\Sigma_2}g_{v(\tau_2)}^{u,\tau_2}\circ g_{v(\tau_1)}^{u,\tau_1}d\tau_2 d\tau_1\right)a,\ \ \text{and}\\
6&\left(q_1\circ\int_{\Sigma_3}g_{v(\tau_3)}^{u,\tau_3}\circ g_{v(\tau_2)}^{u,\tau_2}\circ g_{v(\tau_1)}^{u,\tau_1}d\tau_3d\tau_2 d\tau_1\right)a.
\end{aligned}
\end{equation}

From  formula \eqref{wox} in Proposition~\ref{prop:expbrack} we obtain
\be\label{eq:expsec}
	\begin{aligned}
		\left(d_0^2G_{q_1}^u(v)\right)a=&2\left(q_1\circ\int_{\Sigma_2}g_{v(\tau_2)}^{u,\tau_2}\circ g_{v(\tau_1)}^{u,\tau_1}d\tau_2 d\tau_1\right)a\\ =& \int_{\Sigma_2}\left\langle \lambda,[ g^{u,\tau_2}_{v(\tau_2)},g^{u,\tau_1}_{v(\tau_1)} ](q_1)\right\rangle d\tau_2 d\tau_1  +d_{q_1}^2a \left(\int_0^1 g_{v(t)}^{u,t}(q_1) dt, \int_0^1 g_{v(t)}^{u,t}(q_1) dt\right)=0.
	\end{aligned}
\ee
Indeed, since $v\in\dom(\mc{D}_0^3G_{q_1}^u)$, the second term is zero by \eqref{NUCL}, and moreover 
\be\label{SECO}
	\frac{1}{2}\int_{\Sigma_2}[ g^{u,\tau_2}_{v(\tau_2)},g^{u,\tau_1}_{v(\tau_1)} ](q_1) d\tau_2 d\tau_1\in\IM(d_0 G_{q_1}^u),
\ee
so that the dual product with $\lambda\in\IM(d_0G_{q_1}^u)^\perp$ cancels also the first one. 

By \eqref{wax}, \eqref{NUCL} and $\lambda \in \IM(d_0G_{q_1}^u)^\perp$, for the last term in \eqref{eq:chronexp} we similarly obtain the identity
\be
\begin{aligned}
\left( d_0^3G_{q_1}^u(v) \right)a=&6\left(q_1\circ\int_{\Sigma_3}g_{v(\tau_3)}^{u,\tau_3}\circ g_{v(\tau_2)}^{u,\tau_2}\circ g_{v(\tau_1)}^{u,\tau_1}d\tau_3d\tau_2 d\tau_1\right)a\\=&2\int_{\Sigma_3}\left\langle \lambda,[ g^{u,\tau_3}_{v(\tau_3)},[ g^{u,\tau_2}_{v(\tau_2)},g^{u,\tau_1}_{v(\tau_1)} ]](q_1) \right\rangle d\tau_3 d\tau_2 d\tau_1,
\end{aligned}
\ee
whence the thesis follows.
\end{proof} 

\begin{remark}\label{rem:shuf} The representation formula \eqref{eq:thirdexplicit} for  $\mc{D}_0^3G_{q_1}^u(v)$ in terms of Lie brackets is not unique, and a different representation can be obtained in the following way. 
If we compute the derivative of $\epsilon\mapsto G^u(v\varepsilon)$ according to \cite[\S 2.8, (2.32)]{AgraBook}, we find
\be\label{DALLO}
\frac{\partial}{\partial\varepsilon}G^u(v\varepsilon)=G^u(v\varepsilon)\circ \widetilde{W}(v;\varepsilon), 
\ee
where
\be\label{DILLO}
\widetilde{W}(v;\varepsilon):=\int_0^1\left( \eexp\int_1^t\mathrm{ad}\varepsilon g_{v(\tau)}^{u,\tau}d\tau \right)g_{v(t)}^{u,t}dt.
\ee
Note that the composition order of   $G^u(v;\varepsilon)$ and $\widetilde{W}(v;\varepsilon)$ in \eqref{DALLO}
is reversed compared to \eqref{eq:derivative}. Since, by \eqref{eq:reverse}, we have
\be
	\begin{aligned}
	\eexp\int_1^t\mathrm{ad}\varepsilon g_{v(\tau)}^{u,\tau}d\tau&=\mathrm{Ad}\left( \eexp\int_1^t \varepsilon g_{v(\tau)}^{u,\tau}d\tau \right)\\ &=\mathrm{Ad}\left( -\leexp\int_t^1 \varepsilon g_{v(\tau)}^{u,\tau}d\tau \right)=\eexp\int_t^1-\mathrm{ad}\varepsilon g_{v(\tau)}^{u,\tau}d\tau,
	\end{aligned}
\ee
the expansion in Volterra series of $\widetilde{W}(v;\varepsilon)$ is
	\be\label{eq:Wterms2}
		\widetilde{W}(v;\varepsilon)=\int_0^1 g_{v(t)}^{u,t}dt+\sum_{k=2}^\infty (-\varepsilon)^{k-1}\int_{\Sigma_k}\mathrm{ad}g_{v(\tau_1)}^{u,\tau_1}\circ\dots\circ \mathrm{ad}g_{v(\tau_{k-1})}^{u,\tau_{k-1}}(g_{v(\tau_k)}^{u,\tau_k})d\tau_k\dots d\tau_1.
	\ee
	In \eqref{eq:Wterms2}, the order of the vector fields in the commutator   is  reversed  with respect to \eqref{eq:Wterms}.  
	Our computation  also yields the identity
	\begin{align}
		\int_{\Sigma_3}[ g^{u,\tau_3}_{v(\tau_3)},[ g^{u,\tau_2}_{v(\tau_2)},g^{u,\tau_1}_{v(\tau_1)} ]] d\tau_3 d\tau_2 d\tau_1&=\int_{\Sigma_3}[ g^{u,\tau_1}_{v(\tau_1)},[ g^{u,\tau_2}_{v(\tau_2)},g^{u,\tau_3}_{v(\tau_3)} ]] d\tau_3 d\tau_2 d\tau_1\\ &+\frac{1}{2}\left[ \int_0^1 g_{v(t)}^{u,t} dt,\int_{\Sigma_2}[ g^{u,\tau_2}_{v(\tau_2)},g^{u,\tau_1}_{v(\tau_1)} ] d\tau_2 d\tau_1 \right],
	\end{align}
thanks to which we may obtain another expression for
$\lambda \mc{D} _0^3 G^u_{q_1}(v)$. 	

Even though the representation for the third differential  is not unique,  
for any $v\in \dom(\mc{D}_0^3G^u_{q_1})$ and $\lambda\in \IM(d_0G^u_{q_1})^\perp$ we have the identity:
\be\label{eq:tworep}
\int_{\Sigma_3}\langle \lambda,[ g^{u,\tau_3}_{v(\tau_3)},[ g^{u,\tau_2}_{v(\tau_2)},g^{u,\tau_1}_{v(\tau_1)} ]](q_1)\rangle d\tau_3 d\tau_2 d\tau_1=  \int_{\Sigma_3}\langle \lambda,[ g^{u,\tau_1}_{v(\tau_1)},[ g^{u,\tau_2}_{v(\tau_2)},g^{u,\tau_3}_{v(\tau_3)} ]](q_1)\rangle d\tau_3 d\tau_2 d\tau_1.
\ee
For the second differential, the two series in \eqref{eq:Wterms} and \eqref{eq:Wterms2} produce the same formula, that was already established e.g.~in \cite[\S 20.3]{AgraBook}. 
For further discussions concerning the algebra of all  representations for the $k$th differential we refer to \cite{AGShuffle}.

\end{remark}

	\section{Third order necessary conditions for singular length-minimizers}
	
	\label{QUA}

	We use the Taylor formula  for the end-point map obtained in Section \ref{sec:toan},  
	in connection with our open mapping results, to get third-order necessary conditions satisfied by strictly singular length-minimizers. 
	
	Let $f_1,\ldots, f_k \in \mathrm{Vec}(M)$ be smooth vector fields on the manifold $M$ spanning the distribution $\Delta$ and satisfying the H\"ormander condition \eqref{eq:horm}. We denote by  $X = L^1([0,1];\R^k)$  the space of controls and by $J:X\to [0,\infty)$, $J(u) = \| u\|_{L^1([0,1];\R^k)}$   the length-functional.
	For a fixed $q_0\in M$, we consider the end-point map $F= F_{q_0}: X\to M$. The extended end-point map is the map $\mathcal F: X\to M\times\R$
	given by $
		\mathcal F (u) = (F(u), J(u))$. 
		
	\begin{defi}\label{defi:stricts}	
		A critical point  $u\in X$ of $\mathcal F$ is {\em regular} (resp.~{\em singular}) if there exists $(\lambda,\lambda_0)\in 
	\IM(d_u\mathcal F  )^\perp
	\subset T_{F(u)}^*M\times \R$ such that $\lambda_0\ne 0$ (resp. $\lambda_0= 0$). A critical point  $u\in X$ is {\em strictly} singular if, for every $(\lambda,\lambda_0)\in 
	\IM(d_u\mathcal F  )^\perp$, $\lambda_0=0$. An extremal curve $\gamma$ is regular (resp.~singular, strictly singular), if its associated control $u$ is regular (resp.~singular, strictly singular).
	\end{defi}
	If $u$ is strictly singular, the length-coordinate is  covered by $\IM(d_u\mathcal F )$ and thus the intrinsic second  and third differentials of the extended end-point map $\mathcal F$ coincide with the ones of end-point map $F$ itself.

	Let $ q_1=F_{q_0}(u)$ be the final point and, as in formula \eqref{ABC},   define $G^u_{q_1}: X\to M$ letting $G^u_{q_1}(v) = F_{q_0}(u+v)$.
	In this section we omit in $F$ and $G$ the subscripts $q_0,q_1$, and the superscript $u$.
	The openness of $\mathcal F$ at $u$ is thus further reduced to the openness of $G$ at $0$. By construction, we have the following identities
	\[
	  d_0 G = d_u F,\quad
	  \mc{D}_0^2 G  = \mc{D}_u^2
 F,\quad
 	  \mc{D}_0^3 G  = \mc{D}_u^3
 F.
 \]

	Thanks to Proposition~\ref{prop:thirdiffendp}, given $\lambda\in \IM(d_0G)^\perp$ the trilinear map $\lambda T:\dom(\mc{D}_0^3G)^3  \to\R$ associated with $\lambda \mc{D}_0^3G $ is given by:
	\be\label{eq:trilinear}
			\lambda T (v_1,v_2,v_3)=\frac{1}{3}\sum_{\stackrel{1\le i,j,k\le 3}{i\ne j,j\ne k, i\ne k}}\int_{\Sigma_3}\left\langle \lambda,[g^{u,\tau_3}_{v_i(\tau_3)},[g^{u,\tau_2}_{v_j(\tau_2)},g^{u,\tau_1}_{v_k(\tau_1)}]](q_1) \right\rangle d\tau_3d\tau_2d\tau_1.
	\ee
	Corollary~\ref{cor:condition} specializes as follows.

	\begin{prop}\label{prop:conditionGu}
		Assume that there exists $w_0\in \iso(\mc{D}_0^2G )$ such that:
		\begin{itemize}
		\item[(i)]  
		$\dim(\dom(\mc{D}_0^3G ))+\dim(\IM(G ,2,w_0))>\dim(M)$;
		\item[(ii)] For every non-zero $\lambda\in \IM(G ,2,w_0)^\perp$ and $v\in \dom(\mc{D}_0^3G )$ the real-valued map 
		\[
		\dom(\mc{D}_0^3G )\ni x\mapsto \lambda T  (v,v,x)
		\]
		is not the zero mapping.
		\end{itemize}
		 Then $G $ is open at zero. 
	\end{prop} 
		
		As a consequence we have the following corollary,   that is of interest when
		 $\coker(G,2,w_0)\neq 0$:

	\begin{cor} Let $u$ be the control of a strictly singular length-minimizing curve.   Then, for every $w_0\in \iso(\mc{D}_0^2G)$  
	 one  of the following holds:
		\begin{itemize}
			\item[(i)]  $\dim(\dom(\mc{D}_0^3G))+\dim(\IM(G ,2,w_0))\le \dim(M)$, or
			\item[(ii)] there exist a non-zero covector $\lambda\in \IM(G ,2,w_0)^\perp$ and $v\in \dom(\mc{D}_0^3G )$ such that 
		$
				\lambda T (v,v,x)=0
		$
			for every $x\in \dom(\mc{D}_0^3G )$.
		\end{itemize}
	\end{cor}

	For strictly singular length-minimizers of corank one, the negation of Theorem~\ref{thm:open_cor_one} provides a more refined criterion. Indeed, its contrapositive translates into a pointwise condition as soon as the subspace $\dom(\mc{D}_0^3G )$ is  sufficiently large.

	Let us first recall the  construction of adjoint curves. Let $\gamma:[0,1]\to M$ be an admissible  curve with control $u$, with $\gamma(0) = q_0$ and $\gamma(1) = q_1$. We denote by $P_{t_0}^t	$ the flow of the non-autonomous vector field $V_\tau = f_{u(\tau)}$ as in \eqref{FLUX}. 
	Then we have $\gamma(t) = P_0^t(q_0)$ for $t\in[0,1]$. 
	 By our discussion in 
	Section~\ref{sec:chron_calc},  we see that the differential $(P_t^{1})_*:
	T_{\gamma(t)}M \to  T_{q_1}M$ is given by 
	\[
		(P_t^{1 })_*= \mathrm{Ad} ( (P_t^1)^{-1}) =\mathrm{Ad}(P_1^t)= 
		\mathrm{Ad}\left(\eexp\int_1^t f_{u(\tau)}d\tau\right).
	\]
	The adjoint map $(P_t^{1})^*$ sends $T^*_{q_1}M$ to $T_{\gamma (t)}^*M$. 
	For every  $\lambda\in \IM(d_0G)^\perp$, the curve of covectors defined by 
	\[
	\lambda(t):=(P_t^{1})^*\lambda\in T_{\gamma(t)}^*M,
	\quad  t\in [0,1],
	\]
	is called the \emph{adjoint curve}  to $\gamma$ relative to $\lambda$. In the corank 1 case, this curve is unique  up to normalization of $\lambda \neq 0$.

	\begin{proof}[{\bf Proof of Theorem~\ref{thm:pointwisecondBIS}}] Proving  \eqref{result} 
	is equivalent to show that for any $v_1,v_2,v_3\in\R^k$ and for  $\lambda\in \IM(d_0G)^\perp$ we have 
	\[
		\left\langle \lambda, [ g^{t}_{v_1},[g^{t}_{v_3},g^{t}_{v_2} ]](q_1)\right\rangle+\left\langle \lambda, [ g^{t}_{v_2},[ g^{t}_{v_3},g^{t}_{v_1} ]](q_1)\right\rangle =0,\quad t\in[0,1];
	\]
	where, as in \eqref{eq:gut}, we set $g^{t}_{v}:=(P_t^{1})_*f_v$ for  $t\in [0,1]$ and $v\in \R^k$. Indeed, for all $i,j,\ell=1,\ldots,k$ we have
	\[
	 \langle \lambda(t), [f_i,[f_j, f_\ell]](\gamma(t)) \rangle = \langle (P_t^1)^*  \lambda,   [f_i,[f_j, f_\ell]] (\gamma(t))    \rangle = \langle \lambda, [g_i^t,[g_j^t, g_\ell^t]](q_1) \rangle, 
	\]
where we set $g_i^ t = (P_t^1)_* f_i$.

	Let us fix $\bar{t}\in [0,1)$. 
	Given $s>0$ such that $\bar{t}+s\le 1$, for every   $v\in L^1([0,1];\R^k)$ compactly supported in $(0,1)$ we define
 \be\label{eq:vs}
 	v_s(t):=v\left(\frac{t-\bar{t}}{s}\right),\quad \text{for } t \in [\bar t,\bar t+s], 
\ee
and zero elsewhere.    We consider the subspace of $  \dom(\mc{D}_0^3G )$
 \be\label{eq:subspacees}
 	E_s:=\left\{ u\in \dom(\mc{D}_0^3G )\mid u=v_s\ \ \text{for some }v\in L^1([0,1];\R^k) \text{ with }\int_0^1v(t)dt=0 \right\},
 \ee
 and we observe that while $E_s$ depends on $s$, its codimension does not.
 
Given $v\in L^1([0,1];\R^k)$, its primitive $z\in AC([0,1];\R^k)$ with $z(0)=0$ is 
 \begin{equation} \label{palla}
 	z(t)=\int_0^tv(\tau)d\tau, \quad t\in [0,1].
 \end{equation}
Similarly, for any $v_s$   as in \eqref{eq:vs} let $z_s$   be its primitive with $z_s(0)=0$. It is immediate to establish the identity: 
 \be \label{ZIZ}
 	z_s(t)=sz\left(\frac{t-\bar{t}}{s}\right).
 \ee
 Moreover, if $v_s\in E_s$  the zero-mean property of $v$ translates into:
 \be\label{eq:zerozeromean}
 	z_s(\bar{t})=z_s(\bar{t}+s)=0.
\ee
 
 In the next lines, we shall use several times the following 
  integration by parts formula. For every $0\le\alpha <\beta\le 1$ and $v\in L^1([0,1];\R^k)$, denoting by $z\in AC([0,1];\R^k)$  the primitive   of $v$, we have:
\begin{align}\label{eq:intbypar}
	\int_\alpha^\beta g^{t}_{v(t)}dt&=\int_\alpha^\beta \sum_{i=1}^k g_i^{t}(\dot z_i(t))dt\\
	&=\sum_{i=1}^k g_i^{\beta}z_i(\beta)-\sum_{i=1}^k g_i^{\alpha}z_i(\alpha)-\int_\alpha^\beta \sum_{i=1}^k(\partial_t g_i^{t})z_i(t)dt \\
	&=g^{\beta}_{z(\beta)}-g^{\alpha}_{z(\alpha)}-\int_\alpha^\beta(\partial_t g^{t})_{z(t)}dt.
\end{align}
Starting from  Proposition~\ref{prop:thirdiffendp},
applying this formula to $v_s\in E_s$ and using \eqref{eq:zerozeromean}
 we obtain:
 \begin{equation}
 \label{eq:exps}
 \begin{split} 
 	\frac{1}{2}\lambda \mc{D}_0^3G (v_s)&=\iiint_{\bar{t}\le\tau_3\le\tau_2\le \tau_1 \le \bar{t}+s}\langle\lambda,[g^{\tau_3}_{v_s(\tau_3)},[g^{\tau_2}_{v_s(\tau_2)},g^{\tau_1}_{v_s(\tau_1)}]](q_1)\rangle d\tau_3 d\tau_2 d\tau_1
 	\\
			&=-\iint_{\bar{t}\le\tau_3\le\tau_2\le \bar{t}+s}\langle\lambda,[g^{\tau_3}_{v_s(\tau_3)},[g^{\tau_2}_{v_s(\tau_2)},g^{\tau_2}_{z_s(\tau_2)}]](q_1)\rangle d\tau_3 d\tau_2
			\\
			&
			\quad
			-\iiint_{\bar{t}\le\tau_3\le\tau_2\le \tau_1 \le \bar{t}+s}\langle\lambda,[g^{\tau_3}_{v_s(\tau_3)},[g^{\tau_2}_{v_s(\tau_2)},(\partial_{\tau_1} g^{\tau_1})_{z_s(\tau_1)}]](q_1)\rangle d\tau_3 d\tau_2 d\tau_1
			\\
			&=\int_{\bar{t}\le\tau_2\le\bar{t}+s}\langle\lambda, [g^{\tau_2}_{z_s(\tau_2)},[g^{\tau_2}_{z_s(\tau_2)},g^{\tau_2}_{v_s(\tau_2)}]](q_1)\rangle d\tau_2\\
			&
			\quad
			+\iint_{\bar{t}\le\tau_3\le\tau_2\le \bar{t}+s}\langle\lambda,[(\partial_{\tau_3} g^{\tau_3})_{z_s(\tau_3)},[g^{\tau_2}_{v_s(\tau_2)},g^{\tau_2}_{z_s(\tau_2)}]](q_1)\rangle d\tau_3 d\tau_2\\
			&
			\quad 
			-\iiint_{\bar{t}\le\tau_3\le\tau_2\le \tau_1 \le \bar{t}+s}\langle\lambda,[g^{\tau_3}_{v_s(\tau_3)},[g^{\tau_2}_{v_s(\tau_2)},(\partial_{\tau_1} g^{\tau_1})_{z_s(\tau_1)}]](q_1)\rangle d\tau_3 d\tau_2 d\tau_1
			\\
			&
			= A(s)+B(s)-C(s),
 \end{split}
 \end{equation}
 where $A$, $B$, and $C$ are defined through the last identity.

 From their very definition,  we see that 
the maps 
\be
\tau\mapsto g^{\tau}_i = 
		\mathrm{Ad}\left(\eexp\int_1^t f_{u(\tau)}d\tau\right)f_i
\ee
are Lipschitz continuous for every $i=1,\dots,k$ because their derivatives depend on time through a locally bounded vector field (compare with \eqref{eq:derivativeoffield}).
Then   we have the expansion
 \be \label{LIP}
 	g_i^{\bar{t}+s\theta}=g_i^{\bar{t}}+O(s),
 \ee
 where the error $O(s)$ is uniform for $0\le\theta\le 1$.
 
 Now we estimate the terms $A(s)$, $B(s)$, and $C(s)$ appearing in \eqref{eq:exps}.
 We claim that
 \be
  A(s) = s^3\int_0^1\langle\lambda,[ g^{\bar{t}}_{z(t)},[g^{\bar{t}}_{z(t)},g^{\bar{t}}_{v(t)} ]](q_1)\rangle dt+O(s^4).
 \ee
 To prove this identity we perform in $A(s)$ the change of variable $\tau_2 = \bar t +s t $ with $t\in [0,1]$, and we use  \eqref{ZIZ} and \eqref{LIP}.
 With a similar argument, we show that
 \[
  B(s) = O(s^4)\quad\textrm{and}\quad
  C(s) = O(s^4).
 \]
We conclude that  
 \begin{align}\label{eq:expansionins}
 	\frac{1}{2}\lambda \mc{D}_0^3G(v_s)&= s^3\int_0^1\langle\lambda,[ g^{\bar{t}}_{z(t)},[g^{\bar{t}}_{z(t)},g^{\bar{t}}_{v(t)} ]](q_1)\rangle dt+O(s^4).
 \end{align}

 Let us introduce the set:
 \[
 	\mathfrak{Z}:=\left\{ z\in AC([0,1];\R^k)\mid  \dot z\in \dom(\mc{D}_0^3G^u_{q_1}),\, z(0)=z(1)=0 \right\}.
\] 
As in \eqref{palla}, in the next lines given $z\in \mathfrak{Z}$ we set $v=\dot z$.
By point (i) of Theorem~\ref{thm:open_cor_one}, the map 
$
 	\mathfrak{Z}\ni z\mapsto \lambda\mc{D}_0^3G(\dot z)
$
 is  the zero map. Otherwise the curve $\gamma$ would  not  be length-minimizing. This implies  that the principal term in \eqref{eq:expansionins}, i.e.,~the cubic map  $T:\mathfrak{Z}\to \R$, 
 \[
 	T(z)=\int_0^1\langle\lambda,[ g^{\bar{t}}_{z(t)},[ g^{\bar{t}}_{z(t)},g^{\bar{t}}_{v(t)} ]](q_1)\rangle dt,
 \]
 is identically zero. By polarization, we conclude that the trilinear map $\mathfrak{T}:\mathfrak{Z}\times \mathfrak{Z}\times \mathfrak{Z}\to \R$ associated with $T$, 
\be\label{eq:mcT}
 	\mathfrak{T}(z_1,z_2,z_3)=\frac{1}{6}\sum_{\stackrel{1\le i,j,k\le 3}{i\ne j,j\ne k, i\ne k}}\int_0^1\langle\lambda,[ g^{\bar{t}}_{z_i(t)},[ g^{\bar{t}}_{z_j(t)},g^{\bar{t}}_{v_k(t)}  ]](q_1)\rangle dt,
 \ee
is zero as well. Integrating by parts and using the Jacobi identity, we obtain
 \begin{align}
 	\int_0^1\langle\lambda&,[ g^{\bar{t}}_{z_1(t)},[ g^{\bar{t}}_{z_2(t)},g^{\bar{t}}_{v_3(t)}  ]](q_1)\rangle dt=\\&=-\int_0^1\langle\lambda,[ g^{\bar{t}}_{v_1(t)},[ g^{\bar{t}}_{z_2(t)},g^{\bar{t}}_{z_3(t)}  ]](q_1)\rangle dt-\int_0^1\langle\lambda,[ g^{\bar{t}}_{z_1(t)},[ g^{\bar{t}}_{v_2(t)},g^{\bar{t}}_{z_3(t)}  ]](q_1)\rangle dt\\&=\int_0^1\langle\lambda,[ g^{\bar{t}}_{z_2(t)},[ g^{\bar{t}}_{z_3(t)},g^{\bar{t}}_{v_1(t)}  ]](q_1)\rangle dt- \int_0^1\langle\lambda,[ g^{\bar{t}}_{z_3(t)},[ g^{\bar{t}}_{z_2(t)},g^{\bar{t}}_{v_1(t)}  ]](q_1)\rangle dt
 	\\
 	&
 	\quad
 	+\int_0^1\langle\lambda,[ g^{\bar{t}}_{z_1(t)},[ g^{\bar{t}}_{z_3(t)},g^{\bar{t}}_{v_2(t)}  ]](q_1)\rangle dt,
\end{align}
and a similar expansion holds for $\int_0^1\langle\lambda,[ g^{\bar{t}}_{z_2(t)},[ g^{\bar{t}}_{z_1(t)},g^{\bar{t}}_{v_3(t)} ] ](q_1)\rangle dt$, switching the role of $z_1$ and $z_2$. Plugging these expressions in \eqref{eq:mcT}, we find:
		\be\label{eq:mcT2}
			2\mathfrak{T}(z_1,z_2,z_3)=\int_0^1\langle\lambda,[ g^{\bar{t}}_{z_1(t)},[ g^{\bar{t}}_{z_3(t)},g^{\bar{t}}_{v_2(t)}  ]](q_1)\rangle+ \langle\lambda,[ g^{\bar{t}}_{z_2(t)},[ g^{\bar{t}}_{z_3(t)},g^{\bar{t}}_{v_1(t)}  ]](q_1)\rangle  dt.
		\ee
		To conclude the proof it suffices to show that $\mathfrak{T}=0$  implies that  the trilinear map
		$\mathfrak{A}:\R^k\times \R^k\times \R^k \to \R$
		\begin{align}
			\mathfrak{A}(x,y,z) =\langle\lambda,[ g^{\bar{t}}_{x},[ g^{\bar{t}}_{z},g^{\bar{t}}_{y} ]](q_1)\rangle+ \langle\lambda,[ g^{\bar{t}}_{y},[ g^{\bar{t}}_{z},g^{\bar{t}}_{x}  ]](q_1)\rangle
		\end{align}
		is zero. We now prove by contradiction that if $\mathfrak{T}=0 $  then $\mathfrak{A}=0$, thus completing our argument.
		
		Assume that there exist vectors $x,y,z\in\R^k$ such that $\mathfrak{A}(x,y,z)\ne 0$. We claim that there exists  $\alpha\in C^\infty([0,1];\R)$ 
		such that: 
		\begin{itemize}
			\item[(i)]$\alpha x, \alpha y, \alpha z, \dot\alpha x, \dot\alpha   y$ and $\dot\alpha  z$ belong to $\mathfrak{Z}$, and
			\item[(ii)] we have for some $j\neq 0$
			\begin{equation} \label{alf}
			\int_0^1 \alpha(t)\dot\alpha(t) ^2 dt
			=-4\sqrt{2} j^2\pi^2 .
		\end{equation}
			\end{itemize}
			 To see this, let us consider the standard trigonometric basis of $L^2([0,1];\R)$,
		\[
			\{1\}\cup \left\{ \sqrt{2}\sin(2\pi j t),\sqrt{2}\cos(2\pi j t),\ \ j=1,2,\dots\right\}.
		\]
		Since $\dom(\mc{D}_0^3G)$ is, by assumption, of finite codimension in $\ker(d_0G)$, it is also of finite codimension in $L^2([0,1];\R^k)$, implying that for any given vector $v\in \R^k$ the set 
		\[
			J_v:=\left\{j\in\N \mid \text{either } 
			\sin(2\pi j t) v\not\in\mathfrak{Z}\;\text{or } 
			\left(\cos(2\pi j t)-1\right) v\not\in\mathfrak{Z}\right\}\subset\N
		\]
		is finite.  
		Our assertion follows picking $j\in (J_x\cup J_y\cup J_z)^c$ and defining $\alpha(t):=\sqrt{2}\left(\cos(2\pi j t)-1\right)$. Finally, with $\alpha$ and $j$ chosen as above and using \eqref{alf}, we deduce from equation \eqref{eq:mcT2} that
		\be
			\mathfrak{T}(\alpha x,\alpha y,\dot\alpha  z)=-2\sqrt{2}j^2\pi^2\left(\langle\lambda,[ g^{\bar{t}}_{x},[ g^{\bar{t}}_{z},g^{\bar{t}}_{y} ]](q_1)\rangle+ \langle\lambda,[ g^{\bar{t}}_{y},[ g^{\bar{t}}_{z},g^{\bar{t}}_{x}  ]](q_1)\rangle\right) \neq 0, 
		\ee
		whence the absurd.
	\end{proof}
		
		\begin{remark}\label{rem:compatible} In accordance with the two possible expressions of $\lambda \mc{D}^3_0G$ given in \eqref{eq:tworep}, we observe that \eqref{result} is symmetric with respect to $v_1$ and $v_2$, being therefore independent of the choice of the representation.
		\end{remark}

\section{Third order analysis of a singular extremal}\label{sec:example}

We prove in this section Theorem~\ref{thm:example}. The sub-Riemannian structure $(\R^3,\Delta)$ in its statement has step  $p+1$. 
If $p=1$  then there is no singular curve because if a covector $\lambda$ is orthogonal to $f_1,f_2$, and $[f_1,f_2]$ then it is zero, contradicting the Pontryagin Maximum Principle. 
If $p\ge 2$ then $\Delta$ has constant step equal to $2$ away from the   plane $x_1=0$.   Then any singular extremal  passes through $x_1=0$.

\begin{proof}[{\bf Proof of Theorem~\ref{thm:example} - (i)}]
A horizontal curve 
$\gamma \in AC([0,1];\R^3)$
satisfies, for some control $u\in L^1([0,1];\R^2)$,
\[
\dot{\gamma}_1 =u_1 ,\ \ \dot{\gamma}_2 =u_2 (1-\gamma_1 ),\ \ \dot{\gamma}_3 =u_2 \gamma_1 ^p \quad \textrm{a.e.~on $[0,1]$.}
\]
We assume that $\gamma(0)=0$ and that $|u(t)|=1$ for a.e.~$t\in[0,1]$. If $\gamma$ is a 
strictly singular  
extremal, 
the Prontryagin Maximum Principle yields a non-zero dual curve $\lambda$ satisfying along $\gamma$ the additional equations
\[
0=\langle \lambda , f_1  \rangle = \lambda_1\quad \textrm{and}\quad   0=\langle \lambda ,f_2 \rangle  = \lambda_2(1-\gamma_1)+\lambda_3 \gamma_1 ^p\quad \textrm{on     $[0,1]$.}
\]
The Goh condition $\langle \lambda,[f_1,f_2]\rangle=0$ along $\gamma$
implies the further relation
$\lambda_2 =p\lambda _3 \gamma_1 ^{p-1}$.  
Now, using $|\gamma_1|\leq 1$, it is not difficult to  see that 
$\lambda_2=0$ and thus $\gamma_1=0$. Then we have $u_1=0$ and thus $|u_2|=1$.  
In particular $\gamma(t)=(0,t,0)$ is a singular extremal, whose dual curve is 
constant  $\lambda(t)=(0,0,1)$, and whose control, $u(t)=(0,1)$, is constant  as well.

We claim that the curve  $\gamma (t)=(0,t,0)$ is strictly singular. Indeed,  any regular extremal $\gamma$ together with its dual curve  $\lambda$ is a characteristic
curve of the following Hamiltonian system. Let $H:T^*\R^3=\R^3\times\R^3 \to \R$ be  the Hamiltonian
\[
H(x,\lambda) = \frac12 \big( \langle\lambda  , f_1 (x)\rangle ^2+ \langle \lambda  , f_2 (x)\rangle ^2\big).
\]
If $\gamma$ is regular, then the pair $(\gamma,\lambda)$
solves the system of ordinary differential equations 
$\dot \gamma  = H_\lambda(\gamma ,\lambda)$ and $\dot\lambda = - H_x(\gamma,\lambda)$. In particular, $\gamma$ is smooth.

Now assume by contradiction that for 
the curve  $\gamma (t)=(0,t,0)$  there exists an absolutely continuous curve of covectors $\lambda$ such that $\dot \gamma = H_\lambda(\gamma,\lambda)$ and $\dot\lambda = - H_x(\gamma,\lambda)$.
If $\gamma$ satisfies the first equation, it follows that $\lambda_1 =0$ and $\lambda_2=1$.
From the second equation it then follows that $\dot\lambda_1 = \lambda_2^2=1$, that is a contradiction.
\end{proof}

When $p>2$, we have $[f_1,[f_1,f_2]](\gamma(t))=[f_2,[f_1,f_2]](\gamma(t))=0$, and thus $\gamma$ is not even ``regular abnormal'' in the sense of   \cite[Section 6]{LS}. Now we show that when $p$ is an even integer the singular curve $\gamma$ is locally length-minimizing.

\begin{proof}[{\bf Proof of Theorem~\ref{thm:example} - (ii)}] The precise claim we prove is the following: let $p$ be an even integer and $t_0:=\frac{2}{p+1}$.
For any compact interval $[a,b]\subset \R$ such that $b-a<t_0$,
the segment $\gamma(t)=(0,t,0)$, $t\in[a,b]$, is the unique length-minimizing curve in $(\R^3,\Delta,g)$ joining the point  $(0,a,0)$ to $(0,b,0)$. The proof follows an idea of \cite[Section 7.1]{LS}.

Let $\tau:=b-a<t_0$ and let $\eta:[0,\tau]\to \R^3$ be any horizontal curve parameterized by arc-length such that $\eta(0)=(0,a,0)$ and $\eta(\tau)=(0,b,0)$. Then there exist measurable functions $u_1,u_2$ (unique up to sets of measure zero) such that  
\[
\dot{\eta}_1 =u_1 ,\ \ \dot{\eta}_2 =u_2 (1-\eta_1 ),\ \ \dot{\eta}_3 =u_2 \eta_1 ^p\ \ \text{a.e.~on }[0,\tau].
\]

We claim that, under the hypothesis that $\int_0^\tau u_1 dt= \int_0^\tau u_2 \eta_1 ^p dt=0$, one has 
\begin{equation}   
\label{condiz}
\int_0^\tau u_2 (1-\eta_1 )dt\le \tau,
\end{equation}
with equality holding only for $u_1=0$ and $u_2=1$ a.e.~on $[0,\tau]$. 

Using this claim, the proof can be concluded in the following way. If $\gamma'$ is another admissible curve joining $(0,a,0)$ and $(0,b,0)$ in time $\tau-\epsilon<\tau$, we construct another curve $\gamma''$ by attaching to $\gamma'$ a straight segment back and forth from $(0,b,0)$, for a total time $\epsilon$. But then, the inequality in \eqref{condiz} would be strict, and we have an absurd. Moreover, we observe that  $\mathrm{length}(\gamma)=\tau=\eta_2(\tau)-\eta_2(0)=\int_0^\tau u_2 (1-\eta_1 )dt$. 
So we have equality in \eqref{condiz}, and from the claim it follows that $\eta=\gamma$.
This proves that $\gamma$ is the unique length-minimizing curve between its end-points.

 We prove  
 \eqref{condiz}.	
Let us define the function $V(s)=\int_0^s u_2 dt$ for $s\in [0,\tau]$ and $\beta:=\|\eta_1\|_{L^\infty([0,\tau];\R)}$. Since
\[
\int_0^\tau u_2 (1-\eta_1 )dt\le V(\tau)+\tau \beta=\tau-(\tau-V(\tau))+\beta\tau,
\]
it suffices to show that $\beta\le\frac{\tau-V(\tau)}{\tau}$, that is  a consequence of the following inequalities:
\be\label{eq:ineqprop}
t_0\beta^{p+1} \le \int_0^\tau \eta_1 ^p dt\le \beta^p(\tau-V(\tau)).
\ee
The inequality in the right-hand side follows easily, since 
\[
\int_0^\tau \eta_1 ^p dt=\int_0^\tau \eta_1 ^p(1-u_2 )dt.
\]
To prove the inequality in the left-hand side of \eqref{eq:ineqprop}, we fix $t_{\max}\in [0,\tau]$ such that $|\eta_1(t_{\max})|=\beta$. Since $\eta_1(0)=\eta_1(\tau)=0$ and $|\dot{\eta}_1|\le 1$ a.e.~on $[0,\tau]$, we have  $\min\{t_{\max},\tau-t_{\max}\}\ge \beta$, meaning that the intervals  $I_1=[t_{\max}-\beta,t_{\max}]$ and $I_2=[t_{\max},t_{\max}+\beta]$ are contained in $[0,\tau]$. These arguments also imply that on $I_1$ and $I_2$, $|\eta_1|$ is bounded from below by linear functions $\ell_1$ and $\ell_2$, respectively, such that $\ell_1(t_{\max}-\beta)= \ell_2(t_{\max}+\beta)=0$ and $\ell_1(t_{\max})=\ell_2(t_{\max})=\beta$. Since the exponent $\m$ is even, this implies that $\int_0^\tau \eta_1^p  dt\ge \frac{2}{p+1}\beta^{p+1}$, as desired.

Clearly, if $(u_1  ,u_2 )=(0,1)$ a.e.~on $[0,\tau]$ we have   equality in \eqref{condiz}. Conversely, assuming that equality holds we deduce that $\beta=(\tau-V(\tau))\frac{p+1}{2}$, whence from \eqref{eq:ineqprop} we obtain 
\[
\int_0^\tau \eta_1 ^p(1-u_2 )dt= \beta^p(\tau-V(\tau)),
\]
which holds if and only if $\eta_1  =\beta$ on $[0,\tau]$. Since $\eta_1(0)=0$, then $\eta_1=0$, and this concludes the proof.
\end{proof}

The first case where the previous argument fails is when $p=3$. In this case, to the best of our knowledge  it is   an open question to decide whether the curve  $\gamma(t)=(0,t,0)$ is minimizing or not.
Using Theorem~\ref{thm:open_cor_one},
we will see that the answer is in the  negative.

\begin{proof}[{\bf Proof of Theorem~\ref{thm:example} - (iii)}] Let $F= F_{q_0}:L^1([0,1];\R^2)\to\R^3$ be the end-point map with initial point $q_0=0$ introduced in \eqref{eq:endp}. We claim that $F$ is open at the point  $u=(0,1)\in L^1([0,1];\R^2)$,   the  control of the singular trajectory $\gamma$.
As in \eqref{ABC} 
 and \eqref{lio}, we let $G(v) =G_{q_1}(v) = F_{q_0}(u+v)$, where $q_1=(0,1,0)$. 
 The infinitesimal analysis of $F$ at $u$ is reduced to the infinitesimal analysis of $G$ at $0$. By Lemma \ref{prop:expbrack} the differential of $G$ at $0$ is  given by $d _0 G (v)  =\int_0^1 g^{t}_{v(t)} dt$, where
\[
g^{t}_{v(t)}  =\sum_{i=1}^2 v_i(t)\mathrm{Ad}\left(\eexp\int_1^t f_{u(\tau)}d\tau \right)f_i,
\]
and $\mathrm{Ad}\big(\eexp\int_1^t f_{u(\tau)}d\tau \big)$ is the differential of the inverse of the flow  $(x,t)\mapsto P_1^{t}(x)$, where $x=(x_1,x_2,x_3)\in \R^3$, $t\in [0,1]$, and $P_1^{t}(x)=\gamma(t)$, where $t\mapsto \gamma(t)$ is the horizontal trajectory with control $u$ such that $\gamma(1) = x\in\R^3$. 
Using the formulas for $f_1$ and $f_2$ in \eqref{eq:vfields}, we find
\[
\gamma_1(t)=x_1,\ \ \gamma_2(t)=(t-1)(1-x_1)+x_2,\ \ \gamma_3(t)=(t-1)x_1^3+x_3.
\]
The inverse of the differential   $(P_1^{t})_*^{-1}=(P_t^{1})_*:T_{(0,t,0)}\R^3\to T_{(0,1,0)}\R^3$ is given by
\[
(P_t^{1})_*=\left( \begin{array}{ccc} 1 & 0 & 0 \\ t-1 & 1 & 0\\ -3(t-1)x_1^2 & 0 & 1 \end{array} \right).
\]
Accordingly, the vector fields  $g_
1^{t}$ and $g_
2^{t}$ are
\be
\label{eq:vfpushed}
\begin{split}	
& g_
1^{t}:=(P_t^{1})_*f_1=\frac{\partial}{\partial x_1}+(t-1)\frac{\partial}{\partial x_2}-3(t-1)x_1^2\frac{\partial}{\partial x_3},
\\
& g_2^{t}:=(P_t^{1})_*f_2=f_2,	 
\end{split}
\ee
and we  obtain the following formula for the differential of $G$: 
\[
d_{0}G (v)=
\left(\begin{array}{c} 
\displaystyle \int_0^1v_1(t)dt 
\\ 
\displaystyle 
\int_0^1 \big\{ (t-1)v_1(t)+v_2(t) \big\}
dt 
\\ 0 \end{array}\right).
\]
We then see that a generator of $\operatorname{Im}(d_0 G)^\perp$ is the covector
$\lambda =(0,0,1)$.

We compute the   intrinsic Hessian $\mathcal{D}^2_0 G$, again using  Lemma \ref{prop:expbrack}.
By 	 \eqref{eq:vfpushed},
for every $0\le t_1,t_2 \le 1$ and every $v,w\in \R^2$ at the point   $q_1=(0,1,0)$
we have $ \langle \lambda, [g_{v}^{t_1},g_{w}^{t_2}](q_1)\rangle =0$. Then $\lambda d^2_0G(v)=0$ for every $v\in \ker(d_{0}G)$,  hence $\mathcal{D}^2_0 G=0$.

Finally, we compute the   intrinsic third differential  $\mathcal{D}^3_0G$.
Note first that since the intrinsic Hessian vanishes, by our definition in  \eqref{eq:setsthird} we also have  $\dom(\mc{D}_0^3G)=\ker(d_{0}G)$. 
The only commutator of length three which has non-zero third component is $[g_{1}^{t_1},[g_{1}^{t_2},g_{1}^{t_3}]](q_1)$, and namely we have 
\[
\langle \lambda,[g_{1}^{t_1},[g_{1}^{t_2},g_{1}^{t_3}]](q_1)\rangle=6(t_2-t_3).
\]
Then, again by Lemma \ref{prop:expbrack},   for  $v=(2\pi \sin(2\pi t), 1)\in \ker(d_0G)=\dom(\mc{D}_0^3G)$, the third differential is
\[
\lambda\mc{D}_0^3G(v)=2\int_0^1\int_0^{t_1}\int _0^{t_2} 6v_1(t_1)v_1(t_2)v_1(t_3)(t_2-t_3)dt_3dt_2dt_1=15\neq 0.
\]
Hence, by Theorem \ref{thm:open_cor_one} the mapping $G$ is open at $0$ and thus the singular trajectory $\gamma$ is not optimal (i.e., of minimal length). Alternatively, we could have used Theorem~\ref{thm:pointwisecondBIS} to deduce that $\lambda= 0$, contradicting Pontryagin Maximum Principle.
 
\end{proof}

	\bibliographystyle{abbrv}
	\bibliography{Biblio}

\end{document}